\documentclass[dvips,ssy,preprint]{imsart}

\usepackage{amsthm,amsmath,amsfonts,pstricks,pst-plot}
\RequirePackage[numbers]{natbib}
\RequirePackage[colorlinks,citecolor=blue,urlcolor=blue]{hyperref}

\doi{10.1214/12-SSY085}
\pubyear{2013}
\volume{3}
\issue{1}
\firstpage{262}
\lastpage{321}

\startlocaldefs

\newtheorem{thrm}{Theorem}[section]

\newtheorem{cor}[thrm]{Corollary}
\newtheorem{lemma}[thrm]{Lemma}

\theoremstyle{definition}
\newtheorem{definition}[thrm]{Definition}

\newcommand{\CP}{\mathcal P}
\newcommand{\CV}{\mathcal V}
\newcommand{\CVbn}{\bar{\mathcal V}^n}
\newcommand{\CA}{\mathcal A}
\newcommand{\CAbn}{\bar{\mathcal A}^n}
\newcommand{\CR}{\mathcal R}
\newcommand{\CRb}{\bar{\mathcal R}}
\newcommand{\CRbn}{\bar{\mathcal R}^n}
\newcommand{\CRbm}{\bar{\mathcal R}^m}
\newcommand{\CGbn}{\bar{\mathcal G}^n}
\newcommand{\CGbm}{\bar{\mathcal G}^m}
\newcommand{\CD}{\mathcal D}
\newcommand{\CDbn}{\bar{\mathcal D}^n}
\newcommand{\CB}{\mathcal B}
\newcommand{\E}{\mathbb E}
\newcommand{\N}{\mathbb N}
\newcommand{\bbM}{\mathbb M}
\newcommand{\R}{\mathbb R}
\newcommand{\bw}{{\bf w}}
\newcommand{\der}{\text{d}}
\newcommand{\ptime}{[0,\infty)}
\newcommand{\Rp}{{\mathbb R}_+}
\newcommand{\C}{{\bf C}}
\newcommand{\bL}{{\bf L}}
\newcommand{\bD}{{\bf D}}
\newcommand{\Ebn}{\bar E^n}

\newcommand{\bP}{{\mathbb P}}
\newcommand{\bE}{{\mathbb E}}
\newcommand{\M}{{\bf M}}
\newcommand{\K}{{\bf K}}
\newcommand{\I}{{\bf I}}
\newcommand{\J}{{\bf J}}

\def\wk{\stackrel{w}{\rightarrow}}

\newcommand\z{{\mathcal Z}}
\newcommand{\zn}{{\z}^n}
\newcommand{\bzn}{\bar{\z}^n}
\newcommand{\bzm}{\bar{\z}^m}

\newcommand\dead{{\mathcal D}}
\endlocaldefs

\allowdisplaybreaks

\begin{document}

\begin{frontmatter}

\title{Fluid limits for overloaded multiclass FIFO~single-server
queues with general abandonment}
\runtitle{Overloaded FIFO queues with general abandonment}

\author[a]{\fnms{Otis B.} \snm{Jennings}\ead[label=e1]{otisj@alumni.princeton.edu}}
\and
\author[b]{\fnms{Amber L.} \snm{Puha}\corref{}\ead[label=e2]{apuha@csusm.edu}\thanksref{t2}}
\thankstext{t2}{Research supported in part by UCLA-IPAM}
\address[a]{Otis B.\ Jennings\\
323 E.\ Chapel Hill St\\
 Unit 1571\\
Durham NC, 27702-2469\\
 \printead{e1}}
\address[b]{Amber L.\ Puha\\
Department of Mathematics\\
California State University San Marcos\\
333 S.\ Twin Oaks Valley Road\\
San Marcos, CA 92096-0001\\
\printead{e2}}

\affiliation{California State University San Marcos}

\runauthor{O.B. Jennings and A.L. Puha}

\begin{abstract}
We consider an overloaded multiclass nonidling first-in-first-out
single-server queue with abandonment. The interarrival times, service
times, and deadline times are sequences of independent and identically,
but generally distributed random variables. In prior work, Jennings and
Reed studied the workload process associated with this queue. Under
mild conditions, they establish both a functional law of large numbers
and a functional central limit theorem for this process. We build on that
work here. For this, we consider a more detailed description of the
system state given by $K$ finite, nonnegative Borel measures on the
nonnegative quadrant, one for each job class. For each time and job
class, the associated measure has a unit atom associated with each job
of that class in the system at the coordinates determined by what are
referred to as the residual virtual sojourn time and residual patience
time of that job. Under mild conditions, we prove a functional law of
large numbers for this measure-valued state descriptor. This yields
approximations for related processes such as the queue lengths and
abandoning queue lengths. An interesting characteristic of these
approximations is that they depend on the deadline distributions in
their entirety.
\end{abstract}



\received{\smonth{10} \syear{2012}}

\end{frontmatter}

\section{Introduction}

Here we consider a single-server queue fed by $K$ arrival streams, each
corresponding to
a distinct job class. Upon arrival, each job declares its service time
and deadline requirements.
If a job doesn't enter service within the deadline time of its arrival,
it abandons the queue before
initiating service. Otherwise it remains in queue until it receives its
full service
time requirement. Nonabandoning jobs are served in a nonidling,
first-in-first-out (FIFO)
fashion. The queue is assumed to be overloaded, i.e., the offered load
exceeds one. In addition, it is assumed that the interarrival times,
service times,
and deadline times are mutually independent sequences of independent and
identically, but generally distributed random variables. One aim of
this work is to provide
approximations for functionals such as the queue-length process.

Queueing systems with abandonment are observed in many applications.
Hence, it is a natural phenomenon to
study. For FIFO single-server queues with abandonment, the earliest
analysis focused on models with exponential
abandonment~\cite{ref:AG}, which is not an ideal modeling assumption.
Thereafter, general abandonment
distributions were considered in~\cite{ref:BH}, which restricted to
Markovian service and arrival processes. Shortly thereafter,
stability conditions were determined in~\cite{ref:BBH} without the
Markovian restriction. In the last decade, there has
been considerable progress with analyzing the critically loaded FIFO
single-server queue with general abandonment
\cite{ref:GlW, ref:RW}.
Here we focus on general abandonment for the multiclass overloaded FIFO
single-server model described above.

In~\cite{ref:JR}, Jennings and Reed study the workload process
associated with this
queue under the assumption that the abandonment distributions are
continuous. Note that
the workload process is well defined since each job declares its
deadline upon arrival. In
particular, the value of the workload process is increased by a job's
service time at its arrival
time if and only if the value of the workload process immediately
before the job's arrival is strictly less
than the job's deadline time. As usual, the value of the workload
decreases at
rate one while the workload is positive. Under mild conditions, they
establish both a functional
law of large numbers and a functional central limit theorem for this process.

In this paper, we build on that work. One outcome
is to provide a fluid approximation for the vector-valued
queue-length process. Because the abandonment distributions are not necessarily
exponential, the queue-length vector together with the vector of
residual interarrival times,
class in service, and residual service time does not provide a
Markovian description of the
system state. Additional information about the residual deadline times
is also needed.
We track this information using a measure-valued state descriptor,
which is defined precisely in Section~\ref{sec:sm}. We give an informal
description here.
The state $\z(t)$ of the
system at time $t$ consists of $K$ finite, nonnegative Borel measures
on $\Rp^2$, one for each
class, where $\Rp$ denotes the nonnegative real numbers. Each measure
consists entirely of
unit atoms, one corresponding to each job of that class in the system.
The coordinates of each
atom are determine by two quantities associated with the job. The first
coordinate is the residual virtual sojourn time of that job, which is the
amount of work in the system
(the cumulative residual service times of jobs that don't abandon)
associated with this job and the jobs that arrived ahead of it.
The second coordinate is the residual patience time. This is given by
the job's deadline minus the time in system if it will abandon before
entering service, and is given by
the job's deadline plus its service time minus the time in system
otherwise. Note that the atoms associated
with jobs that will eventually be served
are initially located in $\Rp^2$ above the line of slope one
intersecting the origin, while the atoms
associated with
jobs that won't be served are initially placed on or below this line.
With this description, new jobs
arrive and the corresponding unit atoms are added at the appropriate
coordinates of the system state.
Further, each coordinate of each atom decreases at rate one until the
unit atom reaches one of the
coordinate axes and exits the system. Jobs associated with atoms that
hit the vertical coordinate
axis exit due to service completion. Jobs associated with atoms that
ultimately hit the horizontal
coordinate axis exit due to abandonment (see Figure~\ref{fig:state}).

Our choice of state descriptor is reminiscent of the one used by
Gromoll, Robert, and Zwart~\cite{ref:GRZ}
to analyze an overloaded processor sharing (PS) queue with abandonment.
One distinction is that their work
is for a single-class queue, and therefore has one coordinate rather
than $K$. Another is that the first coordinate
for the unit atom associated with a given job in the PS system is
simply the job's residual service time. We use the residual virtual
sojourn time to determine the first coordinate since it captures the
order of arrivals, which is needed for FIFO.
The evolution of the state descriptor in~\cite{ref:GRZ} is slightly
more complicated than the one used here.
In both cases, the residual patience times decrease at rate one, but in
\cite{ref:GRZ} the service times decrease at rate one
over the number of jobs in the queue. Hence, the nice relationship with
the line of slope one intersecting the origin
present in the FIFO model is not present in the PS model. However, in
both models it is true that hitting a coordinate
axis corresponds to exiting the system with the vertical axis being
associated with service completion and the horizontal
axis being associated with abandonment. So this work provides an
example of how the modeling framework
developed for analyzing PS with abandonment can be adapted to yield an
analysis in another abandonment setting.

Measure valued processes have been used rather extensively for modeling
many server queues with and without abandonment
(see~\cite{ref:KangPang,ref:KangR2010,ref:KangR2012,ref:KaspiR},
and~\cite{ref:PW}).
An important distinction between single-server and many server
first-come-first-serve queues is that the later is not
first-in-first-out. Therefore the measure valued descriptor and
analysis given here is quite different from that found in the many
server queues literature.

We develop a fluid approximation for this measure-valued state
descriptor, which yields
fluid approximations for the queue-length vector and other functionals
of interest.
Similarly to the approximations derived in~\cite{ref:JR} and~\cite{ref:W2006},
these approximations depend on the entire abandonment distribution of
each job class.
We begin by introducing an associated fluid model, which can be viewed
as a formal law of large numbers
limit of the stochastic system. Hence fluid model solutions are
$K$-dimensional measure-valued
functions with each coordinate taking values the space of finite,
nonnegative Borel measures
on $\Rp^2$ that satisfy an appropriate fluid model equation (see
\eqref{fdevoeq2}). We analyze the behavior of fluid
model solutions. Through this analysis we identify a nonlinear mapping
from the fluid workload
to the fluid queue-length vector (see \eqref{eq:fq}). Using the nature
of fluid model solutions,
this mapping is refined to yield the approximations for the number of
jobs of each class in queue
that will and will not abandon (see \eqref{eq:fn} and \eqref{eq:fa}).
In a similar spirit,
we obtain approximations for the number of jobs of each class in system
of a certain age or older
(see \eqref{eq:age1} and \eqref{eq:age2}). In addition, we characterize
the invariant states for this fluid model
(Theorem~\ref{thrm:is}).

Next we justify interpreting fluid model solutions and functionals
derived from them
as first order approximations of prelimit functionals in the stochastic
model by proving a fluid limit theorem
(Theorem~\ref{thrm:flt}). It states that under mild assumptions the
fluid scaled state
descriptors for the stochastic system converge in distribution to
measure-valued functions
that are almost surely fluid model solutions.
A basic element of our fluid limit result is the nature of the scaling employed.
Consistent with the framework in~\cite{ref:JR}, we accelerate both the
arrival process and the service rates, while leaving the abandonment
times unchanged.
One should think of a concomitant speeding up of the server's
processing rate to accommodate the increased customer demand; the
content of the work for any particular customer is the same.
We assume any given customer is unaffected by the increase in the
number of fellow customers, but rather it is the time
in queue that triggers abandonment.
Hence, abandonment propensity does not require adjusting as the other
system parameters increase.

The paper is organized as follows. We conclude this section with a
listing of our notation.
The formal stochastic and fluid models are given Section~\ref{sec:m}.
Then, in Section~\ref{sec:main}, we present the main results (Theorems
\ref{thrm:eu},~\ref{thrm:flt} and~\ref{thrm:is})
and several approximations derived from our fluid model. Theorems~\ref
{thrm:eu} and~\ref{thrm:is} are proved
in Section~\ref{sec:prfthrmeu}.
In the final two sections, we provide the proof
of Theorem~\ref{thrm:flt}. The proof of tightness is presented in
Section~\ref{sec:tightness}, while the
characterization of fluid limit points as fluid model solutions is
presented in Section~\ref{sec:char}.
For this, we prove a functional law of large numbers for a sequence of
measure-valued processes
that we refer to as residual deadline related processes (see Lemma \ref
{lem:ResDeadFLLN}).
Since the residual deadline related processes don't include information
about the service discipline,
the result in Lemma~\ref{lem:ResDeadFLLN} may be of more general
interest for analyzing queues with abandonment.

\subsection{Notation}

The following notation will be used throughout the paper.
Let $\N$ denote the set of strictly positive integers and
let $\R$ denote the set of real numbers.
For $x, y\in\R$, $x\vee y=\max(x, y)$, $x\wedge y=\min(x, y)$, and
$x^+=x\vee0$.
For $x\in\R$, $\lfloor x\rfloor$ denotes the largest integer less than
or equal to $x$ and
$\lceil x\rceil$ denotes the smallest integer greater than or equal to $x$.
For $x\in\R$, $\Vert x \Vert=\vert x\vert$ and for $x\in\R^2$,
$\Vert
x \Vert=\sqrt{x_1^2+x_2^2}$.

The nonnegative real numbers $[0, \infty)$ will be denoted by $\Rp$.
Let $\CB_1$ denote the Borel subsets of $\Rp$
and $\CB_2$ denote the Borel subsets of $\Rp^2$.
On occasion we will use the notation $\Rp^1$ in place of $\Rp$.
For a Borel set $B\in\CB_i$, $i=1, 2$, we denote the indicator function
of the set $B$ by $1_{B}$.
For $i=1, 2$, $B\in\CB_i$ and $x\in\Rp^i$, let $B_x$ denote the
$x$-shift of the set $B$, which is given by
%
%
\begin{equation}\label{def:xshift}
B_x=\{ y \in\Rp^i : y-x\in B\}.
\end{equation}
Given $x\in\Rp$ and $B\in\CB_2$, we adopt the following shorthand notation:
\[
B_x=B_{(x, x)}.
\]
For $i=1, 2$, $B\in\CB_i$, and $\kappa>0$,
let $B^{\kappa}$ denote the $\kappa$-enlargement of $B$, which is
given by
%
%
\begin{equation}\label{def:kappa}
B^{\kappa}=\left\{ x\in\Rp^i : \inf_{y\in B} \Vert x-y\Vert
<\kappa\right
\}.
\end{equation}
Notice that given $i=1, 2$ and $B\in\CB_i$, $(B^{\kappa})_x\subseteq
(B_x)^{\kappa}$,
but $(B^{\kappa})_x$ and $(B_x)^{\kappa}$ are not necessarily the
same set.
In particular, the set resulting from shifting before enlarging may
contain additional points.
We adopt the convention that $B_x^{\kappa}=(B_x)^{\kappa}$, i.e.,
$B_x^{\kappa}$
is the larger of the two sets.

For $i=1, 2$, let $\C_b(\Rp^i)$ denote the set of bounded continuous
functions from $\Rp^i$ to $\R$. Given a finite, nonnegative Borel
measure $\zeta$ on $\Rp^i$, let $\bL(\zeta)$ denote the set of Borel
measurable functions from $\Rp^i$ to $\R$ that are integrable
with respect to $\zeta$. For $g\in\bL(\zeta)$, we define
$\langle g, \zeta\rangle =\int_{{\mathbb R}_+^i}g \der\zeta$ and
adopt the shorthand notation $\zeta(B)=\langle 1_B, \zeta\rangle $ for
$B\in
\CB_i$.
In addition, let $\chi:\Rp\to\R$ be the identify map $\chi(x)=x$ for
all $x\in\Rp$.
Given a finite, nonnegative Borel measure $\zeta$ on $\Rp$, if $\chi
\in
\bL(\zeta)$,
i.e., if $\langle \chi, \zeta\rangle <\infty$, we say that $\zeta$
has a finite
first moment.

For $i=1, 2$, let $\M_i$ denote the set of finite, nonnegative Borel
measures on
$\Rp^i$. The zero measure in $\M_i$ is denoted by ${\bf0}$. For
$x\in
\Rp$,
$\delta_x\in\M_1$ is the measure that puts one unit of mass
at $x$. Similarly, for $x, y\in\Rp$, $\delta_{(x, y)}\in\M_2$ is the
measure that puts one unit of mass at $(x, y)$.
The set $\M_i$ is endowed with the weak topology, that is, for
$\zeta^n, \zeta\in\M_i$, $n\in\N$, we have $\zeta^n\wk\zeta$ as
$n\rightarrow\infty$ if and only if $\langle g, \zeta^n\rangle\rightarrow \langle g, \zeta\rangle $
as $n\rightarrow\infty$, for all $g\in\C_b(\Rp^i)$.
With this topology, $\M_i$ is a Polish space~\cite{ref:P}. Also define
\[
\M_i^K=\{ (\zeta_1, \dots, \zeta_K) : \zeta_k\in\M_i\hbox{ for
}1\le k\le
K\}.
\]
Then $\M_i^K$ endowed with the product topology is also a Polish space.
Given $\zeta\in\M_i^K$ and $g\in\cap_{k=1}^K\bL(\zeta_k)$, we define
the shorthand notation
\[
\left\langle g, \zeta\right\rangle =(\left\langle g, \zeta_1\right\rangle , \dots, \left\langle g,
\zeta_K\right\rangle ).
\]
Given $i=1, 2$ and $\zeta\in\M_i^K$, it will be handy to introduce the
notation $\zeta_+\in\M_i$
for the superposition measure, which is given by
%
%
\begin{equation}\label{def:sp}
\zeta_+(B)=
\sum_{k=1}^K \zeta_k(B), \qquad\hbox{for all }B\in\CB_i.
\end{equation}

Let $\M_{1, 0}$ denote the subset of $\M_1$ containing those measures
that assign zero
measure to the set $\{0\}$. Similarly, let $\M_{2, 0}$ denote the subset
of $\M_2$ containing
those measures that assign measure zero to the set
%
%
\begin{equation}\label{def:C}
C=\Rp\times\{0\}\cup\{0\}\times\Rp.
\end{equation}
For $x, y\in\Rp$, let $\delta_x^+\in\M_{1, 0}$ and $\delta
_{(x, y)}^+\in\M
_{2, 0}$ respectively denote the measures that put a unit mass at the
point $x$ if $x>0$ and $(x, y)$ if $x, y>0$, and are
the zero measure otherwise.
For $i=1, 2$, the set $\M_{i, 0}^K$ is
defined analogously to $\M_i^K$ except that $\zeta_k\in\M_{i, 0}$ for
$1\le k\le K$.
In addition, let $\CB_{1, 0}$ denote those sets in $\CB_1$
that do not contain zero.
Similarly, let $\CB_{2, 0}$ denote those sets in $\CB_2$
that do not meet the set $C$.

Finally, we will use ``$\Rightarrow$'' to denote convergence in
distribution of random
elements of a metric space. Following Billingsley~\cite{ref:B}, we will
use $\bP$ and $\E$
respectively to denote the probability measure and expectation operator
associated with
whatever space the relevant random element is defined on. All
stochastic processes used
in this paper will be assumed to have paths that are right continuous
with finite left limits (r.c.l.l.).
For a Polish space $\mathcal{S}$, we denote by $\bD([0, \infty),
\mathcal{S})$
the space of r.c.l.l.\ functions
from $[0, \infty)$ into~$\mathcal{S}$, and we endow this space with the
usual Skorohod $J_1$-topology
(cf.~\cite{ref:EK}). There are six Polish spaces that will be
considered in this paper: $\R$, $\Rp$,
$\M_1$, $\M_1^K$, $\M_2$, and $\M_2^K$.

\section{The Stochastic and Fluid Models}\label{sec:m}

\subsection{The Stochastic Model}\label{sec:sm}

In this section, we define the model of the GI/GI/1
+ GI queue serving $K$ distinct customer classes, which
will be used for the remainder of the paper.

\paragraph{Initial condition and associated system dynamics}
The initial condition specifies the number $Z_+(0)$ of jobs in the
queue at time zero, as well as the initial virtual sojourn time, initial
patience time and class of each job. Assume that $Z_+(0)$ is
nonnegative integer-valued random variable that is finite
almost surely. The initial virtual sojourn times, initial
patience times
and classes are the first, second, and third coordinates respectively
of the first $Z_+(0)$ elements of the random sequence
$\{(\tilde w_j, \tilde p_j, k_j)\}_{j\in\N}\subset\Rp\times\Rp
\times\{
1, \dots, K\}$. For $1\le j\le Z_+(0)$,
the initial job with initial virtual sojourn time $\tilde w_j$, initial
patience time $\tilde p_j$, and class $k_j$ is called job $j$.

We assume that the elements of the sequence $\{\tilde w_j\}_{j\in\N}$
are finite, positive, and
nondecreasing, which reflects the fact that the service discipline is
first-in-first-out. In particular, the job with the smallest index is
regarded as having arrived to the system before all
other jobs, and therefore has the smallest initial virtual sojourn time.
This is the job currently in service and $\tilde w_1$ represents the time
until its service is completed.
The remaining jobs are waiting in the
queue in order. For $2\le j \le Z_+(0)$, $\tilde w_j$ represents the
amount of time that job $j$
will remain in the system, provided that it doesn't abandon before
entering service.

We assume that the elements of the sequence $\{\tilde p_j\}_{j\in\N}$ are
finite and positive, which reflects the fact that none of the jobs
would be
regarded as having abandoned the queue by time zero.
For $1\le j\le Z_+(0)$ such that $\tilde p_j>\tilde w_j$,
job $j$ is sufficiently patient to wait in the queue until service completion.
For $1\le j\le Z_+(0)$ such that $\tilde p_j\le\tilde w_j$,
job $j$ abandons the queue at time $\tilde p_j$ before entering service.
We assume that $\tilde w_1<\tilde p_1$, which reflects the fact that
job 1
is presently in service and is therefore patient enough to stay in the queue
until service completion. For $2\le j\le Z_+(0)$, we assume that
$\tilde p_j\le\tilde w_j$ if and only if $\tilde w_j=\tilde w_{j-1}$.
When $j$ is such that $\tilde p_j>\tilde w_j$,
one regards $\tilde w_j-\tilde w_{j-1}$ as the service time of job $j$
and this
service time is included in the initial virtual sojourn time of all
jobs $\ell$ such that $j\le\ell\le Z_+(0)$. When $j$ is such that
$\tilde p_j\le\tilde w_j$, a service time for job $j$ is not included
in any of the initial virtual sojourn times.

Since the queue is nonidling and first-in-first-out, all $Z_+(0)$ jobs
in the system at time zero will either abandon or be served by time
$W(0)=\tilde w_{Z_+(0)}$. We
refer to $W(0)$ as the initial workload.\vadjust{\eject}

A convenient way to express the initial condition is to define an
initial random measure $\z(0)\in\M_2^K$. For this, we will find it
convenient to separate the sequence $\{ (\tilde w_j, \tilde p_j, k_j)\}
_{j=1}^{Z_+(0)}$ into $K$ separate
sequences $\{ (\tilde w_{k, j}, \tilde p_{k, j})\}_{j=1}^{Z_k(0)}$, one
for each class. For $1\le k\le K$,
let $Z_k(0)$ denote the number of class $k$ initial jobs. Given
$1\le k \le K$, for $1\le j\le Z_k(0)$, let $i(k, j)$ be the $j$th
smallest index
such that $k_{i(k, j)}=k$ and set $\tilde w_{k, j}=\tilde w_{i(k, j)}$ and
$\tilde p_{k, j}=\tilde p_{i(k, j)}$.
Then, for $1\le k\le K$, let $\z_k(0)\in\M_2$ be given by
\[
\z_k(0)=\sum_{j=1}^{Z_k(0)} \delta_{(\tilde w_{k, j}, \tilde p_{k, j})}^+,
\]
which equals $\mathbf{0}$ if $Z_k(0)=0$. Then let
\[
\z(0)=(\z_1(0), \dots, \z_K(0)).
\]
Our assumptions imply that
$\z(0)$ satisfies
%
%
\begin{equation}\label{eq:InitialConditionFinite}
\mathbf{P}\left(\max_{1\le k\le K}\left\langle 1, \z_k(0)\right\rangle \vee
W(0)<\infty
\right)=1.
\end{equation}
Furthermore,
\[
\mathbf{P}\left(\max_{1\le k\le K}\z_k(0)(C)=0\right)=1.
\]
In particular, $\z(0)\in\M_{2, 0}^K$ almost surely.

\paragraph{Stochastic primitives and associated system dynamics}
The stochastic primitives consist of $K$ exogenous arrival processes
$E_k(\cdot)$, $1\le k\le K$, $K$ sequences of service times $\{
v_{k, i}\}
_{i\in\N}$, $1\le k\le K$,
and $K$ sequences of deadlines $\{d_{k, i}\}_{i\in\N}$, $1\le k\le K$.
We assume that the exogenous arrival processes, the
sequences of service times, and the sequences of deadlines are all
independent of one another.

For a given $1\le k\le K$, the class $k$ arrival process $E_k(\cdot)$
is a rate
$\lambda_k\in(0, \infty)$ renewal process. For $t\in\ptime$, $E_k(t)$
represents the number of class $k$ jobs that arrive to the queue
during the time interval $(0, t]$. We assume that the interarrival times
are strictly positive and denote the sequence of interarrival
times by $\{ \xi_{k, i}\}_{i\in\N}$. Class $k$ jobs arriving after
time zero
are indexed by integers $j>Z_k(0)$. For $t\in\ptime$, let
%
%
\begin{equation}\label{def:A}
A_k(t)=Z_k(0)+E_k(t).
\end{equation}
Then class $k$ job $j$ arrives at time $t_{k, j}=\inf\{t\in\ptime
: A_k(t)\ge j\}$.
Hence, for $j^\prime<j$, $t_{k, j^\prime}\le t_{k, j}$ and we say that
class $k$ job $j^\prime$
arrives before class $k$ job $j$. The inequality is strict for indices
$j>Z_k(0)$. Moreover, for each $j\le Z_k(0)$, $t_{k, j} = 0$.

Given $1\le k\le K$, for each $i\in\N$, the random variable $v_{k, i}$
represents the service time of the $(Z_k(0)+i)$th class $k$ job. That is,
class $k$ job $j>Z_k(0)$ has service time $v_{k, j-Z_k(0)}$. Assume that
the random variables $\{v_{k, i}\}_{i\in\N}$ are strictly positive and
form an independent and identically distributed sequence with finite
positive mean
$1/\mu_k$. Define the class $k$ offered load to be $\rho_k=\lambda
_k/\mu_k$.
We assume that the queue is overloaded. In particular, we assume that
$\rho=\sum_{k=1}^K\rho_k>1$.

Given $1\le k\le K$, for each $i\in\N$, the random variable $d_{k, i}$
represents the deadline of the $(Z_k(0)+i)$th class $k$ job. That is,
class $k$
job $j>Z_k(0)$ has deadline $d_{k, j-Z_k(0)}$. Assume that the random
variables $\{d_{k, i}\}_{i\in\N}$ are strictly positive and form an
independent
and identically distributed sequence of random variables with common
continuous distribution $\Gamma_k$. Assume that the mean $1/\gamma_k$
is finite. Let $F_k(\cdot)$ denote the cumulative distribution function
associated
with $\Gamma_k$, i.e., $F_k(x)=\langle 1_{[0, x]}, \Gamma_k\rangle $
for all $x\in\Rp$.
Denote its complement by $G_k(\cdot)=1-F_k(\cdot)$.

As is the case for jobs in the system at time zero, jobs arriving after
time zero are served in a first-in-first-out, nonidling fashion.
A class $k$ job $j$ arriving to the system after time zero immediately
enters service if the server is available. Otherwise, for class $k$ job $j$
to be served, it must wait until all other jobs currently in the queue
exit via
service completion or abandonment. If class $k$ job $j$ has not entered
service before time $t_{k, j}+d_{k, j-Z_k(0)}$, class $k$ job $j$
abandons the
queue at this time. Otherwise, when class $k$ job $j$ enters service,
it is
served for $v_{k, j-Z_k(0)}$ time units.

It will be convenient to combine the exogenous arrival
process and deadlines into a single measure-valued deadline process.
\begin{definition}\label{def:mvdp} For $1\le k\le K$, the class $k$ deadline
process is given by
\[
\dead_k(t)=\sum_{i=1}^{E_k(t)}\delta_{d_{k, i}}, \quad t\in\ptime.
\]
Then the deadline process is given by
\[
\dead(t)=(\dead_1(t), \dots, \dead_K(t)), \quad t\in\ptime.
\]
\end{definition}
Note that $\dead_k(\cdot)\in\bD(\ptime, \M_1)$ for $1\le k\le K$ and
$\dead(\cdot)\in\bD(\ptime, \M_1^K)$.

\paragraph{The workload process}

The workload process $W(\cdot)\in\bD(\ptime, \Rp)$ tracks as a function
of time, the amount of time needed for
the server to process all jobs currently in the system that will not
abandon. If no additional jobs were to arrive
after time $t$, the system would be empty $W(t)$ time units in the
future. This quantity also records the amount
of time that a newly arriving job would have to wait before being
served, if the job were sufficiently patient.
In particular, as in~\cite{ref:JR}, $W(\cdot)$ almost surely satisfies,
for all $t\in\ptime$,
%
%
\begin{eqnarray}
W(t)&=&W(0)+\sum_{k=1}^K\int_{(0, t]} v_{k, E_k(s)} 1_{\{
d_{k, E_k(s)}>W(s-)\}}\der E_k(s)-B(t), \label{eq:vwp}\\
B(t)&=&\int_0^t 1_{\{ W(s)>0\}}\der s.\label{eq:btp}
\end{eqnarray}
Here $B(\cdot)$ denotes the busy time process. Since the server works
at rate one during any
busy period, the amount of work served by time $t$ is equal to $B(t)$.
Occasionally, it will be convenient to refer to the idle time process,
which is given by
$I(t)=t-B(t)$ for $t\in\ptime$.
For a given time $t$,
the first and second terms in \eqref{eq:vwp} add up all of the work
that enters
the system by time $t$ and doesn't abandon before entering service.
Note that because of the indicator in the integrand, a particular job's
service time is added to the workload if and only if that job will not
abandon before it enters service. Fluid and diffusion limit results for
this process were proved in~\cite{ref:JR},
where it was referred to as the virtual waiting time process. We will
leverage that fluid limit result to carry out the analysis here.

\paragraph{Evolution of the virtual sojourn times and patience times}

For $1\le k\le K$, let
%
%
\begin{eqnarray}
w_{k, j}&=&
\begin{cases}
\tilde w_{k, j}, &1\le j\le Z_k(0), \\
W(t_{k, j}), &j>Z_k(0),
\end{cases}
\nonumber\\
p_{k, j}&=&
\begin{cases}
\tilde p_{k, j}, &1\le j\le Z_k(0),
\\ d_{k, j-Z_k(0)} + v_{k, j-Z_k(0)}1_{\{d_{k, j-Z_k(0)}>W(t_{k, j}-)\}
}, &j>Z_k(0).
\end{cases}
\nonumber
\end{eqnarray}
Note that for each $1\le k\le K$ and $j>Z_k(0)$,
\[
W(t_{k, j})=W(t_{k, j}-) + v_{k, j-Z_k(0)} 1_{\{
d_{k, j-Z_k(0)}>W(t_{k, j}-)\}}.
\]
Hence, if class $k$ job $j$ is served, $w_{k, j}$ includes
the job's service time in addition to the time spent waiting for
service to begin.
So, if class $k$ job $j$ is served, it stays in the system $w_{k, j}$
time units.
Therefore, $w_{k, j}$ is referred to as the virtual sojourn time of
class $k$ job $j$.
Further, $p_{k, j}$ represents the initial patience of class $k$ job $j$.
If a class $k$ job $j$ arriving after time zero will enter service before
time $t_{k, j}+d_{k, j-Z_k(0)}$, the job's patience time is taken to be
$d_{k, j-Z_k(0)}+v_{k, j-Z_k(0)}$ to account for the fact that it will
not abandon
$d_{k, j-Z_k(0)}$ time units after arrival. Instead, it will stay until time
$t_{k, j}+w_{k, j}$ when it receives its full service time requirement.
Otherwise, if class $k$ job $j$ won't enter service before the deadline
expires, the job will abandon at time $t_{k, j}+d_{k, j-Z_k(0)}$, and
$p_{k, j}=d_{k, j-Z_k(0)}$.
Then for each $1\le k\le K$ and $j\in\N$, the sojourn time $s_{k, j}$ of
class $k$
job $j$ is given by
\[
s_{k, j}=\min(w_{k, j}, p_{k, j}).
\]
This quantity indicates precisely how long class $k$ job $j$ will
reside in the system.

For all $t\in\ptime$, $1\le k\le K$, and $1\le j\le A_k(t)$, define
%
%
\begin{eqnarray}
w_{k, j}(t) &=&(w_{k, j}-(t-t_{k, j}))^+, \label{eq:rst}\\
p_{k, j}(t)&=&(p_{k, j}-(t-t_{k, j}))^+.\label{eq:rpt}
\end{eqnarray}
For $1\le k\le K$, $j\in\N$ and $t\ge t_{k, j}$, $w_{k, j}(t)$ and
$p_{k, j}(t)$
respectively represent the residual virtual sojourn time and residual
patience time of
class $k$ job $j$ at time $t$. Then the residual sojourn time $s_{k, j}(t)$
for class $k$ job $j$ at time $t\ge t_{k, j}$ is given by
\[
s_{k, j}(t)=\min(w_{k, j}(t), p_{k, j}(t)).
\]

\paragraph{Measure-valued state descriptor}
For $1\le k\le K$, define the class $k$ state descriptor
by
%
%
\begin{equation}\label{eq:zk}
\z_k(t)=\sum_{j=1}^{A_k(t)}\delta^+_{(w_{k, j}(t), p_{k, j}(t))},
\qquad
t\in\ptime.
\end{equation}
The state descriptor is defined as
%
%
\begin{equation}\label{eq:z}
\z(t)=(\z_1(t), \dots, \z_K(t)), \qquad t\in\ptime.
\end{equation}
For each $1\le k\le K$, $\z_k(\cdot)\in\bD(\ptime, \M_2)$, and
$\z(\cdot)\in\bD(\ptime, \M_2^K)$.

Figure~\ref{fig:state} depicts one component of a hypothetical
system state at a fixed time. The points in the figure correspond
to unit atoms of the measure.
All points move to the left and
down at rate one. The dotted diagonal line $p=w$ separates the points
in the figure into two groups. The jobs associated with points
on or below the line will eventually abandon; the points above
the line represent jobs that will be served.
Once a point reaches one of the coordinate axes it immediately
leaves the system and so is no longer included in the system state.
Among all of the components of the system state, there is a unique point
that has the smallest residual virtual sojourn time coordinate
and lies above the diagonal. This point corresponds to the job
currently in service.

Notice that in Figure~\ref{fig:state} there are three sets of points
that are
aligned vertically.
In each set, at most one of these points is above the
dotted line and the corresponding job will be served.
For each residual virtual sojourn time assumed by some job that is in
the system
at this fixed time, there is exactly one of these jobs for which the location
of the corresponding point in the component of the state descriptor associated
with that job's class lies above the diagonal. This job will be served
and it
arrived before any other job with the same residual virtual sojourn time.
These later arriving jobs aren't patient enough to remain in queue
until service begins.
Instead each will abandon. Therefore the corresponding
points in the components of the state descriptor associated
with those jobs' classes lie on or below the diagonal.
Although their sequence of arrivals relative to one another is not captured
by the state descriptor, the relative residual patience times
reveal the order in which they will abandon.

%
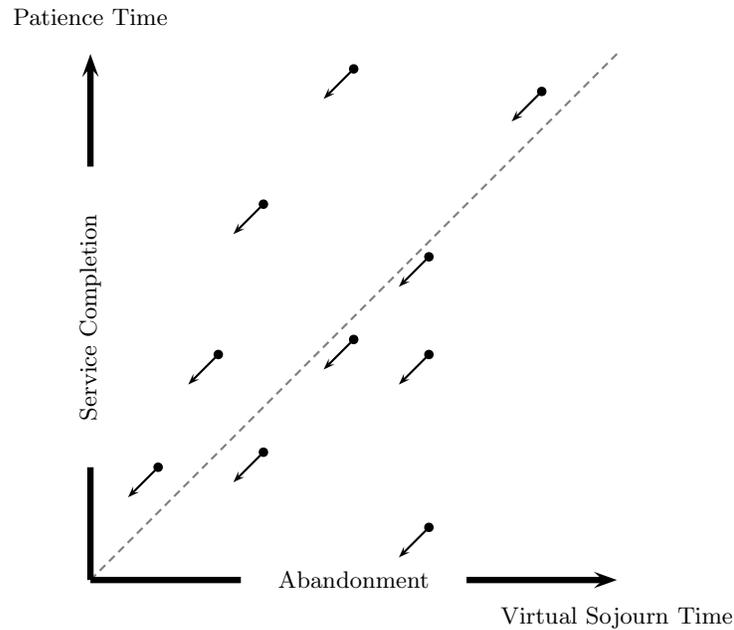
\begin{figure}[t]
\begin{center}
\begin{pspicture}(-1, -1)(8, 8)
\psline[linestyle=dashed, dash=3pt 2pt, linecolor=gray](0, 0)(7, 7)
{\psset{linewidth=3\pslinewidth}

\psline{-}(0.0, 0.0)(2.0, 0.0)
\psline{->}(5.0, 0.0)(7.0, 0.0)
\psline{-}(0.0, 0.0)(0.0, 1.5)
\psline{->}(0.0, 5.5)(0.0, 7.0)}

\rput(3.5, 0.0){Abandonment}
\rput{90}(0.0, 3.5){Service Completion}

\psline{*->}(6.0, 6.5)(5.6, 6.1)

\psline{*->}(3.5, 6.8)(3.1, 6.4)
\psline{*->}(3.5, 3.2)(3.1, 2.8)

\psline{*->}(4.5, 4.3)(4.1, 3.9)
\psline{*->}(4.5, 3.0)(4.1, 2.6)
\psline{*->}(4.5, 0.7)(4.1, 0.3)

\psline{*->}(2.3, 5.0)(1.9, 4.6)
\psline{*->}(2.3, 1.7)(1.9, 1.3)

\psline{*->}(1.7, 3.0)(1.3, 2.6)

\psline{*->}(0.9, 1.5)(0.5, 1.1)

\rput(0, 7.5){Patience Time}
\rput(7, -0.5){Virtual Sojourn Time}
\end{pspicture}
\caption{One coordinate of a hypothetical system state at a fixed time
with depicted transitions.} \label{fig:state}
\end{center}
\end{figure}

The state descriptor satisfies the following system of dynamic equations.
For each $1\le k\le K$ and for all $B\in\CB_{2, 0}$ and $t\in\ptime$,
%
%
\begin{equation} \label{eq:dynamics1}
\z_k(t)(B)
=
\sum_{j=1}^{A_k(t)}
1_{B_{t-t_{k, j}}}(w_{k, j}, p_{k, j}).
\end{equation}
To see this, simply note that for each $B\in\CB_{2, 0}$, $x\in\Rp$, and
$w, p>0$,
\[
((w-x)^+, (p-x)^+)\in B
\quad\Leftrightarrow\quad
(w-x, p-x)\in B
\quad\Leftrightarrow\quad
(w, p)\in B_x.
\]
The dynamic equation \eqref{eq:dynamics1} is equivalent to
%
%
\begin{equation}\label{eq:dynamics1prime}
\z_k(t)(B)
=
\z_k(0)(B_t) +
\sum_{j=1+Z_k(0)}^{A_k(t)}
1_{B_{t-t_{k, j}}}(w_{k, j}, p_{k, j}).
\end{equation}
Given $B\in\CB_{2, 0}$ and $x, y\in\Rp$, notice that
\[
(B_x)_y=B_{x+y}.
\]
This together with \eqref{eq:dynamics1} implies that,
for each $1\le k\le K$ and for all $B\in\CB_{2, 0}$ and $h, t\in
\ptime$,
%
%
\begin{equation} \label{eq:dynamics2}
\z_k(t+h)(B)
=
\z_k(t)(B_h)
+
\sum_{j=A_k(t)+1}^{A_k(t+h)}
1_{B_{t+h-t_{k, j}}}(w_{k, j}, p_{k, j}).
\end{equation}

\subsection{The Fluid Model}\label{sec:fms}

In this section, we define the fluid model associated with this
GI/GI/1+GI queue with $K$ distinct job classes. The primitive data for
this fluid model consists of the
vector $\lambda$ of arrival rates, the vector $\mu$ of service rates,
and the vector $\Gamma$ of deadline distributions. The triple
$(\lambda
, \mu, \Gamma)$ is referred to as supercritical data since $\rho>1$.
We begin by summarizing the results in~\cite{ref:JR} that suggest a
workload fluid model. Then, we develop
a full measure-valued fluid model. For this, we first need to define
the scaling regimes that yield the desired limiting dynamics.

\paragraph{The Sequence of Time Accelerated Systems}
We consider a sequence of systems indexed by $n \in\N$
in which the arrival rates and mean service times in the $n$th
system are sped up by a factor of $n$. We use a superscript
$n$ to denote all processes and parameters associated with
the $n$th system.

Specifically, the interarrival times $\{\xi^n_{k, i}\}_{i \in\N}$
for class $k$ in the $n$th system are given by
$\xi^n_{k, i} = \xi_{k, i} / n$ for $i\in\N$. Then $E_k^n(\cdot)$
denotes the class $k$ exogenous arrival processes in
the $n$th system. Hence, for $1\le k\le K$, $E_k^n(t)=E_k(nt)$
for $t\in\ptime$ and $\lambda_k^n=n\lambda_k$.
Similarly, the class $k$ service times $\{v^n_{k, i}\}_{i \in\N}$ in the
$n$th system are given by $v^n_{k, i} = v_{k, i} / n$ for $i\in\N$. Then
$\mu_k^n=n\mu_k$ and $\rho_k^n=\rho_k$. Hence, $\rho^n=\rho$.

The deadline sequence is unscaled. In particular, for each $1\le k\le
K$ and
$n, i\in\N$, $d_{k, i}^n=d_{k, i}$. Hence, we omit the superscript
when referring
to the class $k$ deadlines, distribution $\Gamma_k$, mean $1/\gamma
_k$, and
cumulative distribution function $F_k$ or its complement $G_k$.
Then the class $k$ deadline process for the $n$th system is given by
\[
\dead_k^n(t)=\sum_{i=1}^{E_k^n(t)} \delta_{d_{k, i}}, \qquad t\in
\ptime,
\]
and $\dead^n(\cdot)=(\dead_1^n(\cdot), \dots, \dead_K^n(\cdot))$.

For each $n\in\N$, there is an initial condition $Z^n(0)$ and
$\{(\tilde w_j^n, \tilde p_j^n, k_j^n)\}_{j\in\N}$ satisfying the conditions
indicated above. The initial conditions may vary with $n$.
Then for $n\in\N$, $1\le k\le K$, and $t\in\ptime$,
$A_k^n(t) = Z_k^n(0) + E_k^n(t)$.
If we suppose for simplicity that $Z_+(0)=0$ in the unscaled system,
then the job arrival times $\{t^n_{k, j}\}_{j \in\N}$
for class $k$ jobs in the $n$th system would be given by
\[
t^n_{k, j} =
\begin{cases}
0, &1\le j\le Z_k^n(0), \\
\frac{t_{k, j-Z_k^n(0)}}{n}, &j>Z_k^n(0).
\end{cases}
\]

The workload $W^n(\cdot)$, busy time $B^n(\cdot)$, and idle time
$I^n(\cdot)$ processes for the $n$th system satisfy equations analogous
to \eqref{eq:vwp} and \eqref{eq:btp}.
The residual virtual sojourn $w_{k, j}^n(\cdot)$ and residual patience
$p_{k, j}^n(\cdot)$ times, $1\le k\le K$ and $j\in\N$, for the $n$th
system are defined as in \eqref{eq:rst} and \eqref{eq:rpt}. The class
$k$ state descriptor $\z_k^n(\cdot)$ and the state descriptor $\z
^n(\cdot)$ for the $n$th system are defined as in \eqref{eq:zk} and
\eqref{eq:z}. Analogs of \eqref{eq:dynamics1} and \eqref{eq:dynamics2}
hold for the $n$th system as well.

\paragraph{The Workload Fluid Model}
For the time accelerated scaling regime, the authors of~\cite{ref:JR}
identify what one might
refer to as a workload fluid model. Namely, define a \textit{workload
fluid model
solution} to be a function $w:\ptime\to\Rp$ satisfying
%
%
\begin{equation}\label{eq:flwp}
w(t)=w(0)+\sum_{k=1}^K\rho_k \int_0^tG_k(w(s))\der s -t, \qquad t\in
\ptime.
\end{equation}
Equation \eqref{eq:flwp} can be interpreted as a fluid analog of
\eqref
{eq:vwp} and \eqref{eq:btp}.
Notice that any solution to \eqref{eq:flwp} is necessarily Lipschitz
continuous with Lipschitz constant
$\rho-1$, and therefore is almost everywhere differentiable. In \cite
{ref:JR}, it is asserted that for each $w_0\in\Rp$
a unique workload fluid model solution $w(\cdot)$ with $w(0)=w_0$
exists. They further show that workload fluid model
solutions satisfy a nice monotonicity property \cite[Theorem
1]{ref:JR}. To describe this, let
\begin{eqnarray*}
w_l&=&\sup\left\{ u\in\Rp: \sum_{k=1}^K\rho_k G_k(u) < 1\right\}
, \\
w_u&=&\sup\left\{ u\in\Rp: \sum_{k=1}^K\rho_k G_k(u) \le1\right
\}.
\end{eqnarray*}
Note that by continuity of $G_k(\cdot)$ for each $k$ and the fact that
$\rho>1$, $0<w_l\le w_u<\infty$.
They show that for a workload fluid model solution $w(\cdot)$ such that
$w(0)\not\in[w_l, w_u]$,
$w(\cdot)$ is strictly monotone and
%
%
\begin{equation}\label{eq:fdwl}
\lim_{t\to\infty}w(t)
=
\begin{cases}
w_l, &\hbox{if } w(0)<w_l, \\
w_u, &\hbox{if } w(0)>w_u.
\end{cases}
\end{equation}
Otherwise, if $w(0)\in[w_l, w_u]$, then $w(t)=w(0)$ for all $t\in
\ptime$.

Workload fluid model solutions also satisfy a useful relative ordering property.
In particular, given two workload fluid model solutions $w_1(\cdot)$
and $w_2(\cdot)$
such that $w_1(0)\le w_2(0)$ it follows that for all $t\in\ptime$,
%
%
\begin{equation}\label{eq:RelOrd}
w_1(t)\le w_2(t).
\end{equation}
To see this note that if $w_1(0)\le w_u$ and $w_l\le w_2(0)$, then
\eqref{eq:RelOrd}
follows immediately from the monotoncity property of workload fluid
model solutions.
Otherwise, $w_2(0)<w_l$ or $w_1(0)>w_u$. If $w_1(0)>w_u$, then by
continuity and
\eqref{eq:fdwl} there exists $t\in\ptime$ such that $w_2(t)=w_1(0)$.
For $s\in\ptime$,
set $w(s)=w_2(t+s)$. Then, it is straightforward to verify that
$w(\cdot
)$ is a workload
fluid model solution with $w(0)=w_1(0)$. Hence,
by uniqueness of fluid model solutions, $w(\cdot)=w_1(\cdot)$ so that
$w_2(\cdot+t)=w_1(\cdot)$. Then for all
$s\in\ptime$, $w_1(s)=w_2(s+t)\le w_2(s)$. An analogous argument demonstrates
\eqref{eq:RelOrd} when $w_2(0)<w_l$.

In~\cite{ref:JR}, the authors justify interpreting the fluid model as a
first order approximation
of the stochastic system by showing that if $W^n(0)\to w_0$ almost
surely as $n\to\infty$,
where $w_0$ is a finite nonnegative constant, then
$(W^n, I^n)\Rightarrow(w(\cdot), 0)$,
as $n\to\infty$, where $w(\cdot)$ is the unique workload fluid model
solution such that $w(0)=w_0$
\cite[Theorem~1]{ref:JR}. Here, we wish to cite a slightly modified
version of \cite[Theorem~1]{ref:JR},
which we state next. Before this, we make note of a representation for
the workload process developed
in~\cite{ref:JR} that will be of use here. By \cite[(19) and arguments
presented in the proof of Theorem~1]{ref:JR},
for all $n\in\N$ and $0\le s\le t<\infty$,
%
%
\begin{eqnarray}\label{eq:WorkRep}
W^n(t)&=&W^n(s) + I^n(t)-I^n(s) - (t-s) \\
&&+X^n(t)-X^n(s) +\sum_{k=1}^K \rho_k \int_s^t G_k(W^n(u))\der
u, \nonumber
\end{eqnarray}
where $\{X^n(\cdot)\}_{n\in\N}\subset\bD(\ptime, \R)$ is such
that as
$n\to\infty$,
%
%
\begin{equation}\label{eq:Xn}
X^n(\cdot)\Rightarrow0.
\end{equation}

\begin{thrm}\label{thrm:JR}
If $W^n(0)\Rightarrow W_0$ as $n\to\infty$, then as $n\to\infty$,
%
%
\begin{equation}\label{eq:JR}
(W^n, I^n)\Rightarrow(W^*(\cdot), 0),
\end{equation}
where $W^*(0)$ is equal in distribution to $W_0$. Furthermore,
$W^*(\cdot)$ is almost surely a workload fluid model solution.
\end{thrm}

\begin{proof}[Summary of proof]
Let $\iota:\ptime\to\ptime$ be given by $\iota(t)=t$ for all $t\in
\ptime$.
Since the limit in \eqref{eq:Xn} is deterministic, it follows that as
$n\to\infty$,
\[
W^n(0)-\iota(\cdot) + X^n(\cdot)\Rightarrow W_0-\iota(\cdot).
\]
Using tools developed in~\cite{ref:RW}, the authors further show that
for all $n\in\N$,
%
%
\begin{equation}\label{eq:CMT}
(W^n(\cdot), I^n(\cdot))=(\varphi_G, \psi_G)\left(W^n(0)-\iota
(\cdot
)+X^n(\cdot)\right),
\end{equation}
where $(\varphi_G, \psi_G):\bD(\ptime, \R)\to\bD(\ptime, \Rp^2)$
is a Lipschitz
continuous function in the topology of uniform convergence on compact sets.
Then, by the continuous mapping theorem, it follows that as $n\to
\infty$,
%
%
\begin{equation}\label{eq:CMT2}
(W^n(\cdot), I^n(\cdot))\Rightarrow(\varphi_G, \psi_G)\left
(W_0-\iota
(\cdot)\right).
\end{equation}
The authors of~\cite{ref:JR} go on to show that for
$w_0\in\Rp$ $(\varphi_G, \psi_G)(w_0-\iota(\cdot))=(w(\cdot),
0)$, where
$w(\cdot)$ is the unique workload fluid model solution such that $w(0)=w_0$.
\end{proof}

\paragraph{The Measure-Valued Fluid Model}
To develop a measure-valued fluid model, we add spatial scaling to the
sequence of
time accelerated systems. In particular in order to obtain a fluid
scaled system, we divide
space in the $n$th time accelerated system by a factor of $n$. Then,
for $n\in\N$,
$1 \le k \le K$, and $t\in\ptime$,
\[
\Ebn_k(t)=\frac{E^n_k(t)}{n}, \qquad\CDbn_k(t)=\frac{\CD^n_k(t)}{n},
\qquad\hbox{and}\qquad\bzn_k(t)=\frac{\z^n_k(t)}{n}.
\]
The $n$th fluid scaled state descriptor in the sequence is given by
\[
\bzn(\cdot)
=
(\bzn_1(\cdot), \ldots, \bzn_K(\cdot)).
\]
We are interested the the limiting behavior as $n$ approaches infinity.

One can speculate as to how the dynamics of the system will behave
in the limit. Class $k$ fluid should arrive to the system at rate
$\lambda_k$
and be distributed over the positive quadrant as it enters. Class $k$
fluid arriving at
time $t$ should have virtual sojourn time $w(t)$, where $w(\cdot)$ is
an appropriate fluid
workload solution, and patience time distributed according to $\Gamma
_k$. Then the residual virtual sojourn and residual patience times
should each decrease at rate one
until at least one is zero, at which time the fluid would exit from the
system. In particular,
at time $t+h$, fluid that arrived to $(w, p)$ at time
$t$ should be at $(w-h, p-h)$ if $(w-h)\wedge(p-h)>0$ and should have
exited the system otherwise.
Such departures would be associated with abandonment if $p-h\le0\wedge
(w-h)$ and service
completion if $w-h\le0$ and $w-h < p-h$.
We wish to capture these dynamics by defining an appropriate fluid model.

The initial measure will need to satisfy various natural conditions.
To define these, first recall the definition of $C$ given in \eqref{def:C}.
Then, given $\vartheta\in\M_2^K$, recall that $\vartheta_+\in\M_2$
denotes the superposition
measure (see \eqref{def:sp}). Let
\[
w_{\vartheta}=\sup\{ x\in\Rp: \vartheta_+([x, \infty)\times\Rp
)>0\}.
\]
Define $\I\subset\M_2^K$ to be the collection of $\vartheta\in\M_2^K$
such that
\begin{enumerate}
\item[(I.1)] $\vartheta_+(C_x)=0\hbox{ for all }x\in\Rp^2$,
\item[(I.2)] $w_{\vartheta}<\infty$, 
%
\item[(I.3)] $\max_{1\le k\le K} G_k(w_{\vartheta
}-\varepsilon)>0\hbox{ for all }\varepsilon>0$.
\end{enumerate}

We refer to the collection ${\mathcal C}=\{ C_x : x\in\Rp^2\}$ as the
corner sets. Then
condition (I.1) is that each coordinate of the initial measure doesn't
charge corner sets,
which is equivalent to the condition that each coordinate of the
initial measure doesn't charge vertical or horizontal lines. A
motivation for requiring (I.1) is to
prevent exiting mass from resulting in discontinuities. Since (I.1) it
is preserved by the
fluid model dynamics (see Property~2 of Theorem~\ref{thrm:eu}), it is a
natural to require it of the initial condition.

Given $\vartheta\in\M_2^K$ that satisfies (I.1)
and $\varepsilon>0$ there exists $\kappa>0$ such that
%
%
\begin{equation}\label{eq:InitialCorner}
\max_{1\le k\le K} \sup_{x\in\Rp^2} \vartheta_k(C_x^{\kappa
})<\varepsilon.
\end{equation}
To see this, given $x\in\Rp^2$, let $p_i(x)=x_i$ for $i=1, 2$.
Then, given $\nu\in\M_2$ and $i=1, 2$, let $\pi_i(\nu)\in\M_1$ be
the projection measure such that for all $f\in\C_b(\Rp^1)$
%
%
\begin{equation}\label{def:pi}
\left\langle f, \pi_i(\nu)\right\rangle =\left\langle f\circ p_i, \nu\right\rangle .
\end{equation}
Then, if $\vartheta\in\M_2^K$ satisfies (I.1), it follows that for
$i=1, 2$, $\pi_i(\vartheta_+)$
doesn't charge points. Hence, for every $\varepsilon>0$ there exists
$\kappa>0$ such that
\[
\max_{i=1, 2}\sup_{y\in\Rp} \pi_i(\vartheta_+)([(y-\kappa
)^+, y+\kappa
])<\frac{\varepsilon}{2},
\]
(cf.\ \cite[Lemma A.1]{ref:GPW}).
Since for each $1\le k\le K$ and $x\in\Rp^2$,
\[
\vartheta_k(C_x^{\kappa})
\le\vartheta_+(C_x^{\kappa})
\le\pi_1(\vartheta_+)([(x_1-\kappa)^+, x_1+\kappa])+ \pi
_2(\vartheta
_+)([(x_2-\kappa)^+, x_2+\kappa]),
\]
\eqref{eq:InitialCorner} follows.

Condition (I.2) dictates that the fluid analog of the initial workload
is finite and
condition (I.3) requires that the initial fluid workload does not
exceed the maximal deadline.
In particular, let
\[
d_{\max}=\max_{1\le k\le K}\sup\{ x\in\Rp: G_k(x)>0\}.
\]
Then, by (I.2), $w_{\vartheta}<d_{\max}$ if $d_{\max}=\infty$.
Further, by (I.3), $w_{\vartheta}\le d_{\max}$ if $d_{\max}<\infty$,
which is natural. Indeed in the stochastic system
the workload can only exceed the maximal deadline by at most
one service time, i.e., once the workload has jumped above the
maximal deadline, all incoming jobs necessarily abandon until
such time as the maximal deadline exceeds the workload.
This discrepancy vanishes in the fluid limit.
The reader will see that it is needed mathematically when
$d_{\max}<\infty$ to prevent fluid from building up on a moving
vertical line during $(0, w_{\vartheta}-d_{\max}]$, which
would prevent (I.1) from being preserved.

A \textit{fluid model solution} for the initial measure $\vartheta\in\I$
is a function
$\zeta:\ptime\to\M_2^K$ such that $\zeta(0)=\vartheta$ and
for each $1\le k\le K$, $B\in\CB_2$, and $t\in\ptime$, $\zeta_k$ satisfies
%
%
\begin{equation}\label{fdevoeq2}
\zeta_k(t)(B)
=
\zeta_k(0)(B_t)
+
\lambda_k \int_0^t \left(\delta_{w(s)}^+\times\Gamma_k\right)(B_{t-s})
\der s,
\end{equation}
where $w(\cdot)$ denotes the unique workload fluid model solution with
$w(0)=w_\vartheta$.
Notice that \eqref{fdevoeq2} is a fluid analog of \eqref{eq:dynamics1}.

\section{Main Results and Approximation Formulas} \label{sec:main}

Here we state the two main theorems proved in this paper (Theorems \ref
{thrm:eu} and~\ref{thrm:flt}),
which together validate approximating various performance processes via
fluid model solutions.
Then, we derive some specific approximation formulas. Lastly, we
identify the set of invariant
states for the fluid model.

\begin{thrm}\label{thrm:eu} Let $\vartheta\in\I$. Then there exists a
unique fluid model solution
$\zeta(\cdot)$ for the data $(\lambda, \mu, \Gamma)$ such that
$\zeta
(0)=\vartheta$.
In addition,
\begin{enumerate}
\item$w_{\zeta(t)}=w(t)$ for all $t\in\ptime$,
where $w(\cdot)$ is the unique workload fluid model solution such that
$w(0)=w_{\vartheta}$;
\item$\zeta_+(t)(C_x)=0$ for all $x\in\Rp^2$ and $t\in\ptime$;
\item$\zeta(\cdot)$ is continuous.
\end{enumerate}
\end{thrm}

The proof of this theorem is given in Section~\ref{sec:prfthrmeu}.
Property~1 is that the right edge of the support
of the fluid model solution superposition measure is equal to the
workload fluid model solution for all time,
which is true by definition at time zero. It implies that (I.2) holds
for all time. Then by the monotoncity
properties of workload fluid model solutions, (I.3) holds for all time
as well. The proof of Property~1 is fairly
straightforward, as the reader will see in Section~\ref{sec:prfthrmeu}.
Property~2 is that the fluid model solution superposition
measure doesn't charge corner sets for all time, i.e., that (I.1) holds
for all time, and its proof is more involved. For this,
we first prove Lemma~\ref{prop:NearBoundary}, from which Property~2
follows by letting $\varepsilon$
decrease to zero. Lemma~\ref{prop:NearBoundary} is then used together
with some additional arguments
to verify Property~3, continuity.

The next result justifies regarding fluid model solutions as
first order approximations for the measure-valued state descriptor of
the original stochastic system.

\begin{thrm}\label{thrm:flt}
Suppose that $\z_0^*=(\z_{0, 1}^*, \dots, \z_{0, K}^*)$ is a random measure
in $\M_2^K$ with superposition measure $\z_{0, +}^*$ such that
\begin{itemize}
\item[(A.1)] $\bP( \z_0^*\in\I)=1$,
%
\item[(A.2)] $\bE[ \langle p_1+p_2, \z_{0, +}^*\rangle ]<\infty$,
\item[(A.3)] $\bE[ \z_{0, +}^*(\Rp^2)]<\infty$.
\end{itemize}
Let $W_0^*=w_{\z_0^*}$. Also suppose that, as $n\to\infty$,
%
%
\begin{eqnarray}\label{eq:IC}
&&\left(\bzn(0), \left\langle p_1, \bzn(0)\right\rangle , \left\langle p_2, \bzn
(0)\right\rangle
, W^n(0)\right) \\
&&\qquad\Rightarrow\left(\z_0^*, \left\langle p_1, \z_0^*\right\rangle , \left
<p_2, \z
_0^*\right\rangle
, W_0^*\right).\nonumber
\end{eqnarray}
Then, as $n\to\infty$,
\[
\bzn(\cdot)\Rightarrow\z^*(\cdot),
\]
where $\z^*(0)$ is equal in distribution to $\z_0^*$.
Furthermore, $\z^*(\cdot)$ is almost surely a fluid model solution.
\end{thrm}
This result is proved via verifying tightness and then proving that any
limit point is almost surely a fluid model solution, which are done in
Sections~\ref{sec:tightness} and~\ref{sec:char} respectively.

Next, we use the fluid model to derive various approximation formulas.
These include approximation formulas that demonstrate the nonlinear
nature of
state space collapse that takes place for this model. That is to say,
one can recover
various $K$ dimensional quantities such as the fluid approximation for
the queue-length
vector from the fluid approximation for the workload process. However,
this mapping
depends on the entirety of the deadline distributions and is therefore
nonlinear in its
dependence on the workload process. See \eqref{eq:fq}, \eqref{eq:fn},
and \eqref{eq:fa}
below. In addition, we are able to approximate various age related
processes such
as the number of jobs in the system that are within a specific age range.

To describe these, we will need to introduce the following functional.
Given $\vartheta\in\I$,
let $\zeta(\cdot)$ denote the unique fluid model solution such that
\mbox{$\zeta(0)=\vartheta$}.
Let $w(\cdot)$ denote the unique workload fluid model solution such
that $w(0)=w_{\vartheta}$.
Fluid that arrives at time $s>0$ has fluid workload $w(s)$. Some of
this fluid remains in the system
at time $t\ge s$ only if
\[
w(s)-(t-s)>0.
\]
Recall that $w(\cdot)$ is continuous. Furthermore, it is strictly
decreasing if $w(0)>w_u$
and bounded above by $w_u$ otherwise. Additionally, since
$\rho>1$, $w_u<d_{\max}$. Finally, by (I.3), $w(0)\le d_{\max}$.
Hence, $w(s)<d_{\max}$ for all $s>0$. Therefore, $w'(s)=\sum_{k=1}^K
\rho_kG_k(w(s))-1>-1$
for all $s>0$. Then, as a function of $s\in\ptime$, $w(s)+s$ is
continuous and strictly increasing.
Furthermore, for $s\in[0, t]$, the values range from $w(0)$ to
$w(t)+t\ge t$ at time $t$. Hence,
$\inf\{ 0\le s\le t : w(s)+s\ge t\}$ is well defined, finite, and can
be interpreted as the time at which fluid
departing at time $t$ via service completion arrived to the system. For
$t\in\ptime$, let
%
%
\begin{equation}\label{eq:tau}
\tau(t)=
\inf\{ 0\le s\le t : w(s)+s\ge t\}.
\end{equation}
Then $\tau(t)=0$ for all $t\in[0, w(0)]$, so that $w(\tau(t))+\tau(t)>
t$ for all $t\in[0, w(0))$.
Further, for $t\ge w(0)$,
%
%
\begin{equation}\label{eq:tauprop}
w(\tau(t))+\tau(t)=t.
\end{equation}
For $t>w(0)$, all fluid in the system at time $t$ arrived during the
time interval $(\tau(t), t]$.

\paragraph{Nonlinear State Space Collapse}
Given $\vartheta\in\I$, let $\zeta(\cdot)$ denote the unique fluid
model solution with $\zeta(0)=\vartheta$
and $w(\cdot)$ denote the unique workload fluid model solution with
$w(0)=w_{\vartheta}$.

\textit{Queue-Length Vector Fluid Approximation.}
For $1\le k\le K$ and $t\in\ptime$,
%
%
\begin{equation}\label{eq:fq}
z_k(t)
\equiv
\zeta_k(t)(\Rp^2)
=
\begin{cases}
\zeta_k(0)((\Rp^2)_t)
+
\lambda_k \int_0^{t} G_k(a)\der a
, & t< w(0), \\
\lambda_k \int_0^{w(\tau(t))} G_k(a)\der a
, & t\ge w(0).
\end{cases}
\end{equation}
We have chosen $a$ as the variable of integration to suggest the
interpretation age, or time in system.
This provides a simple interpretation of this formula. At time
$t<w(0)$, none of the fluid that entered the
system after time zero has been fully processed. So the integral term
only needs to address departures
resulting from expiring deadlines. Then, fluid arriving in $(0, t]$
remains in the system at time $t$
if and only if the initial deadline exceeds the current age.
Hence the simple form of the integral.
At time $t\ge w(0)$, all fluid initially in the system has departed,
and $w(\tau(t))$ can be interpreted as the age of the fluid departing
the system via service completion at time $t$. So, for $t\ge\tau(t)$,
$w(\tau(t))$ can be
thought of as the age of the oldest fluid in the system at time $t$.

\textit{Nonabandoning Jobs Queue-Length Vector Fluid Approximation.}
Let $U=\{ (w, p)\in\Rp^2 : w<p\}$.
Then the mass present in $U$ at time $t$ is associated with fluid that
doesn't abandon.
For $1\le k\le K$ and $t\in\ptime$,
%
%
\begin{equation}\label{eq:fn}
n_k(t)
\equiv
\zeta_k(t)(U)
=
\begin{cases}
\zeta_k(0)(U_t)
+
\lambda_k \int_0^{t}G_k(w(v)) \der v
, & t<w(0), \\
\lambda_k \int_{\tau(t)}^t G_k(w(v)) \der v
, & t\ge w(0).
\end{cases}
\end{equation}
Here the variable of integration should be regarded as time, and the
formula has the following
interpretation. Fluid that has not departed the system by time $t$ will
not abandon prior to service
completion if the initial deadline exceeds the fluid workload at the
time of arrival.

For $t\ge w(0)$, we can rewrite this result in an alternative form that
highlights the independence
of the service time distributions to the deadline and interarrival time
distributions. For $1\le k\le K$ and
$t\ge w(0)$, let
\[
w_k(t)=\rho_k \int_{\tau(t)}^t G_k(w(v)) \der v.
\]
Then $w_k(t)$ denotes the amount of fluid workload in the system at
time $t$ due to class $k$ fluid.
For each $1\le k\le K$, we have for all $t\ge w(0)$,
\[
n_k(t)=\mu_kw_k(t).
\]
In this form, the formula is reminiscent of linear state space collapse,
but the dependence of $w_k(\cdot)$ on $w(\cdot)$ is nonlinear.

\textit{Abandoning Jobs Queue-Length Vector Fluid Approximation.}
Let $L=\{ (w, p)\in\Rp^2 : p\le w\}$.
Then the mass present in $L$ at time $t$ is associated with fluid that abandons.
For $1\le k\le K$ and $t\in\ptime$, let $a_k(t)=\zeta_k(t)(L)$.
Then, for
$1\le k\le K$ and $t\in\ptime$,
%
%
\begin{equation}\label{eq:fa}
a_k(t)
=
\begin{cases}
\zeta_k(0)(L_t)
+
\lambda_k \int_0^{t} \left(G_k(t-v)-G_k(w(v))\right) \der v
, & t<w(0), \\
\lambda_k \int_{\tau(t)}^t \left(G_k(t-v)-G_k(w(v))\right)\der v
, & t\ge w(0).
\end{cases}
\end{equation}
Note that it is easy to verify that $z_k(t)=a_k(t)+n_k(t)$ for $1\le
k\le K$ and $t\in\ptime$.

\paragraph{Age Related Fluid Approximations}
Given $\vartheta\in\I$, let $\zeta(\cdot)$ denote the unique fluid
model solution with $\zeta(0)=\vartheta$
and let $w(\cdot)$ denote the unique workload fluid model solution with
$w(0)=w_{\vartheta}$.
Fluid in the system at time $t>0$ that arrived at time $s\in(0, t]$ is
$t-s$ units old at time $t$.
Hence, for fluid in the system at time $t>0$ that arrives in $(0, t]$ to
be of age at least
$u$, it must have arrived by time $t-u$. Then $s\in(0, t-u]$. This fluid
has residual offered
waiting time $(w(s)-(t-s))^+$. Since it is in the system at
time $t$, $w(s)-(t-s)>0$,
i.e., $w(s)+s>t$ so that $s\in(\tau(t), t-u]$. For $t\ge w(0)$, this is
an empty time interval if
$w(\tau(t))\le u$. Otherwise, for $t\ge w(0)$ and $s\in(\tau(t), t-u]$,
$w(s)-(t-s)\in(0, w(t-u)-u]$.
For $t\ge w(0)$ and $0\le u\le t$, let
\[
H(t, u)=[0, \left(w(t-u)-u\right)^+]\times\Rp.
\]
This is the line $\{0\}\times\Rp$ if $u\ge w(\tau(t))$ and is a
vertical stripe otherwise.
Then, for $1\le k\le K$, $t\ge w(0)$, and $0\le u\le t$, we can interpret
\[
z_k(t, u)\equiv\zeta_k(t)(H(t, u)),
\]
to be the amount of class $k$ fluid in the system at time $t$ of age at
least $u$.

If $t<w(0)$, some initial fluid may remain in the system at time $t$.
We regard that fluid as being $t$
units old. At time $t$, the residual offered waiting time of such fluid
lies in $[0, w(0)-t]$. If we consider
a time $0\le u\le t<w(0)$ and ask for the total amount of fluid of age
at least $u$, this would also include
fluid that arrived after time zero and by time $t-u$. By \eqref
{eq:flwp}, $w(t-u)-u>w(0)-t$. Hence the
definition of $H(t, u)$ and interpretation of $z_k(t, u)$ naturally
extend to all $0\le u\le t<w(0)$.
Then, for $1\le k\le K$ we have the following: for $0\le u\le t<w(0)$,
\begin{align}\label{eq:age1}
z_k(t, u)&=
\zeta_k(0)(H(t, u)_t)
+
\lambda_k \int_u^t G_k(v)\der v, \\
\intertext{and for $t\ge w(0)$, }
z_k(t, u)&=
\begin{cases}
\lambda_k \int_u^{w(\tau(t))} G_k(v)\der v
, & 0\le u< w(\tau(t)), \\
0, & w(\tau(t))\le u\le t.\label{eq:age2}
\end{cases}
\end{align}
One can obtain similar approximations for the abandoning and
nonabandoning queue length of a certain age or older by
computing the measure of an appropriately chosen set.

\paragraph{Invariant States}
An \textit{invariant state} is a measure $\theta\in\I$ such that the
unique fluid model solution $\zeta(\cdot)$
with initial measure $\theta$ satisfies $\zeta(t)=\theta$ for all
$t\in
\ptime$. Here we identify the collection of invariant states.

For this, we begin by recalling some measure theoretic background.
Details can be found in \cite[Chapter~1]{ref:F}. Let
\[
{\mathcal R}=\left\{ [a, b)\times[c, d) \subset\Rp^2 : a, c\in\Rp,
a\le
b\le\infty\hbox{ and } c\le d\le\infty\right\}.
\]
Then ${\mathcal R}$ is an elementary family (a collection of sets that
contains the emptyset, is closed under
pairwise intersection, and such that complements of members of the collection
can be written as finite unions of members of the collection).
Further note that $\sigma({\mathcal R})=\CB_2$. Let
\[
{\mathcal R}'=\left\{\cup_{l=1}^{m} R_l : m\in\N\hbox{ and }R_l\in
{\mathcal R}\right\}.
\]
Then ${\mathcal R}'$ is an algebra (a collection of sets that contains
the emptyset and is closed under
pairwise union and relative complementation). In fact, ${\mathcal R}'$
is the algebra generated by ${\mathcal R}$. It is easy to see that any
finite pre-measure
(additive $\Rp$ valued function that assigns value zero to the
emptyset) on ${\mathcal R}$ extends to a finite pre-measure on
${\mathcal R}'$.
Then, by the Carath\'{e}odory extension theorem \cite[Theorem
1.14]{ref:F}, any finite pre-measure defined on ${\mathcal R}'$
uniquely extends to a finite Borel measure on $\Rp^2$. Thus, in order
to uniquely specify a finite Borel measure on $\Rp^2$, it suffices to
specify a finite pre-measure on~${\mathcal R}$.

Given $w_l\le w\le w_u$, let $\theta^w\in\I$ be the unique
finite Borel measure on $\Rp^2$ that for $1\le k\le K$ satisfies
\[
\theta_k^w([w, \infty)\times\Rp)=0,
\]
and for $0\le a<b\le w$ and $0\le c<d\le\infty\in\Rp$,
\[
\theta_k^w([a, b)\times[c, d))
=
\lambda_k \int_{w-b}^{w-a}\Gamma_k([c+u, d+u))\der u.
\]
It is easy to verify that these relationships determine a finite
pre-measure on ${\mathcal R}$.
Indeed, since $[w, \infty)\times\Rp$ has measure zero, it suffices to
specify each
$\theta_k^w$ on sets in ${\mathcal R}$ that don't meet $[w, \infty
)\times
\Rp$.
Define
\[
\J=\{ \theta^w : w_l\le w\le w_u\}.
\]

\begin{thrm}\label{thrm:is} The set of invariant states is given by
$\J$.
\end{thrm}

The proof of Theorem~\ref{thrm:is} is given in Section~\ref{sec:prfthrmeu}.
Note that for $w_l\le w\le w_u$ and $1\le k\le K$,
the fluid approximations for queue length $z_k^w$ and nonabandoning
queue length
$n_k^w$ take the form
\begin{eqnarray*}
z_k^w&\equiv&\theta_k^w(\Rp^2)=\lambda_k\int_0^w G_k(u)\der u, \\
n_k^w&\equiv&\theta_k^w(U)=\lambda_k w G_k(w).\\
\end{eqnarray*}

\section{Properties of Fluid Model Solutions}\label{sec:prfthrmeu}

In this section, we prove Theorems~\ref{thrm:eu} and~\ref{thrm:is}.
The proof of existence and uniqueness for Theorem~\ref{thrm:eu}
is relatively straightforward. Indeed, since the right hand side of
\eqref{fdevoeq2} only depends on $(\lambda, \mu, \Gamma)$ and
$\vartheta$, \eqref{fdevoeq2} can be regarded as a definition,
provided that the integral term is a well defined function taking
values in $\M_2^K$. In this regard, note that $\vartheta_k\in\M_2$
and $\delta_{w(s)}^+\times\Gamma_k\in\M_2$ for all $s\in\ptime$
and $1\le k\le K$. Hence, the main issue is to show that the integral
is well defined on all of ${\mathcal B}_2$, which is demonstrated
here using the Carath\'{e}odory extension theorem and
Dynkin's $\pi\lambda$-theorem.

\begin{proof}[Proof of Existence and Uniqueness for Theorem~\ref{thrm:eu}]
Fix \mbox{$\vartheta\in\I$}.
First we verify existence of a fluid model solution with initial
measure $\vartheta$.
For this, fix $1\le k\le K$ and $0<t<\infty$.
Given $B\in{\mathcal R}$, we have that $B=[a, b)\times[c, d)$ for some
$a, c\in\Rp$, $a\le b\le\infty$, and $c\le d\le\infty$. Define
$f:[0, t]\to[0, 1]$ by
\[
f(s)=\left(\delta_{w(s)}^+\times\Gamma_k\right)(B_{t-s}).
\]
By \eqref{eq:fdwl}, the fact that $w_l>0$ and monotonicity properties
of $w(\cdot)$, $w(s)>0$ for all $s>0$.
Then, since $t>0$, we have that for $s\in(0, t]$,
\[
f(s)=\left(\delta_{w(s)}\times\Gamma_k\right)(B_{t-s}).
\]
Therefore, for $s\in[0, t]$,
\[
f(s)
=
\begin{cases}
G_k(c+t-s)-G_k(d+t-s), &\hbox{if }a+t-s\le w(s)< b+t-s, \\
0, &\hbox{otherwise.}
\end{cases}
\]
We see that $a+t-s\le w(s)< b+t-s$ if and only if $a+t\le w(s)+s< b+t$
if and only if
$\tau(a+t)\le s<\tau(b+t)$, where $\tau(\infty)=\infty$. Hence, for
$s\in[0, t]$,
\[
f(s)
=
\begin{cases}
G_k(c+t-s)-G_k(d+t-s), &\hbox{if } \tau(a+t) \le s <\tau(b+t), \\
0, &\hbox{otherwise.}
\end{cases}
\]
Then, since $G_k(\cdot)$ and $\tau(\cdot)$ are continuous, $f$ is Borel
measurable and\break
$\int_0^tf(s)ds$ is well defined. Further, since $\tau(\cdot)$ is
monotone increasing,
\[
\int_0^t f(s)\der s
=
\int_{t\wedge\tau(a+t)}^{t\wedge\tau(b+t)} \left(
G_k(c+t-s)-G_k(d+t-s)\right)\der s.
\]

For $B\in{\mathcal R}$, define
\[
\gamma_k(t)(B)=\vartheta_k(B_t)+\lambda_k\int_0^t \left(\delta
_{w(s)}^+\times\Gamma_k\right)(B_{t-s})\der s.
\]
It is clear that $\gamma_k(t)(\emptyset)=0$.
Further, $\gamma_k(t)(\Rp^2)\le\vartheta_k(\Rp^2)+\lambda_k
t<\infty$.
So, in order to verify that $\gamma_k(t)$ is a finite premeasure on
${\mathcal R}$, it suffices to verify countable additivity.
This amounts to demonstrating that the summation and integral can be
interchanged, which is an immediate consequence
of the monotone convergence theorem. Then, by the Carath\'{e}odory
extension theorem, $\gamma_k(t)$
extends to a finite Borel measure $\gamma_k^*(t)$ on ${\mathcal B}_2$.

We must verify that $\gamma_k^*(t)$ satisfies \eqref{fdevoeq2} for all
$B\in{\mathcal B}_2$. For this, we
use Dynkin's $\pi\lambda$-theorem \cite[Chapter~1
Theorem~3.3]{ref:B}. Let
%
%
\begin{equation}\label{def:Pi}
{\mathcal P}=\{ [a, \infty)\times[c, \infty) : 0\le a, c<\infty\}.
\end{equation}
This is a $\pi$-system since it is closed under intersection. Let
\[
{\mathcal L}=\left\{ B\in{\mathcal B}_2 :
\gamma_k^*(t)(B)=\vartheta_k(B_t)+\lambda_k\int_0^t \left(\delta
_{w(s)}\times\Gamma_k\right)(B_{t-s})\der s
\right\}.
\]
Then ${\mathcal L}$ is a $\lambda$-system since $\Rp^2\in{\mathcal L}$,
${\mathcal L}$
is closed under countable unions (by the monotone convergence theorem), and
$A\setminus B\in{\mathcal L}$ whenever $B, A\in{\mathcal L}$ and
$B\subset A$. Further,
\[
{\mathcal P}\subset{\mathcal R}\subset{\mathcal L}\subset{\mathcal B}_2.
\]
Then, since the $\sigma$-algebra generated by ${\mathcal P}$ is
${\mathcal B}_2$,
it follows from Dynkin's $\pi\lambda$-theorem that ${\mathcal
L}={\mathcal B}_2$.

Since $1\le k\le K$ and $t>0$ were arbitrary, it follows that $\gamma
^*:\ptime\to\M_2^K$,
which is given by
\[
\gamma^*(t)=(\gamma_1^*(t), \dots, \gamma_K^*(t)),
\]
is a fluid model solution. Uniqueness is immediate since any fluid
model solution $\zeta$
such that $\zeta(0)=\vartheta$ satisfies $\zeta_k(t)(B)=\gamma
_k^*(t)(B)$ for all
$B\in{\mathcal R}$ (see \cite[Theorem~1.14]{ref:F}).
\end{proof}

\begin{proof}[Proof of Property~1 for Theorem~\ref{thrm:eu}]
Fix $\vartheta\in\I$. Let $w(\cdot)$ be the unique workload fluid model
solution with $w(0)=w_{\vartheta}$
and $\zeta(\cdot)$ be the unique fluid model solution with $\zeta
(0)=\vartheta$.
Since $\zeta(0)=\vartheta$ and $w(0)=w_{\vartheta}$, $w_{\zeta(0)}=w(0)$.
Fix $t\in(0, \infty)$ and $\varepsilon\in(0, w(t))$.
Set $B=[w(t)-\varepsilon, \infty)\times\Rp$. Then $B_{(2\varepsilon
, 0)}=[w(t)+\varepsilon, \infty)\times\Rp$.
We wish to show that $\zeta_+(t)(B)>0$ and $\zeta
_+(t)(B_{(2\varepsilon, 0)})=0$.
Then, since $\varepsilon\in(0, w(t))$ is arbitrary, it follows that
$w_{\zeta(t)}=w(t)$.

Suppose that $\zeta_+(t)(B_{(2\varepsilon, 0)})>0$. Then by \eqref{fdevoeq2},
there exists $s\in[0, t]$ such that $w(t)+\varepsilon+t-s\le w(s)$, i.e.,
$w(t)+t+\varepsilon\le w(s)+s$. But $w(\cdot)+\iota(\cdot)$ is
strictly increasing,
so that $w(t)+t+\varepsilon> w(s)+s$, which is a contradiction. Thus,
$\zeta_+(t)(B_{(2\varepsilon, 0)})=0$.

Since $w(\cdot)+\iota(\cdot)$ is continuous and strictly increasing,
there exists
$s_1\in[0, t)$ such that $w(t)+t-\varepsilon\le w(s_1)+s_1$. Then, for
all $s\in(s_1, t]$,
$w(t)+t-\varepsilon\le w(s)+s$. Let $s_2=(t-d_{\max})^+$. Then
$s_2\in[0, t)$
and $t-s<d_{\max}$ for all $s\in(s_2, t]$. Let $s^*= s_1 \vee s_2$.
Then $s^*<t$ and
\[
\sum_{k=1}^K \int_{0}^t \left( \delta_{w(s)}^+ \times\Gamma_k
\right
)(B_{t-s})\der s
\ge
\sum_{k=1}^K \int_{s_1}^t G_k(t-s)\der s
\ge
\sum_{k=1}^K \int_{s^*}^t G_k(t-s)\der s
>
0.
\]
Hence, by \eqref{fdevoeq2}, $\zeta_+(t)(B)>0$.
\end{proof}

The next goal is to verify Properties~2 and~3 for Theorem~\ref{thrm:eu}.
For this, we state and prove the following lemma. Property~2 follows
immediately from Lemma~\ref{prop:NearBoundary} by letting $\varepsilon$
decrease to zero. To verify Property~3, we must show that fluid model
solutions change very little over
short time intervals. Note that there are three mechanisms that cause
the measure-valued function
to change: fluid arriving, fluid departing, and the measure evolving. Verifying
that the fluid departs in a smooth way is the main issue that needs to
be addressed. We will
use the result in the following lemma to assist with this as well.

\begin{lemma}\label{prop:NearBoundary}
Given $\vartheta\in\I$, let $\zeta(\cdot)$ be the unique fluid
model solution
such that $\zeta(0)=\vartheta$. For every $T, \varepsilon>0$, there exists
$\kappa>0$ such that
\[
\max_{1\le k\le K}\sup_{t\in[0, T]}\sup_{x, y\in\Rp} \zeta
_k(t)(C_{(x, y)}^{\kappa})<\varepsilon.
\]
\end{lemma}

To ease the reader's efforts to follow the proof of Lemma \ref
{prop:NearBoundary}
given below, we outline the basic strategy. Having a good understanding
of this deterministic
argument will assist the reader in following the stochastic
generalization used to prove
Lemma~\ref{lem:BndReg}. The basic idea in this case is to use \eqref
{fdevoeq2} on a given
$\kappa$ enlargement of a corner set. Then \eqref{eq:InitialCorner} can
be used to bound
the contribution from mass present in the system at time zero. Next one
determines
the time interval during which mass must arrive in order to contribute
to the vertical portion of the
enlarged corner set. As the reader will see, the end points of this
time interval can be expressed
in terms of the function $\tau(\cdot)$. Then one obtains a bound on the
mass that can be present
in the horizontal portion of the enlarged corner set. This second part
turns out to be fairly easy to
bound, as the reader will see. Hence the main difficultly is to bound
the amount of mass that falls
in the vertical portion. This relies on bounding the length of the
aforementioned time interval, which
can be done by demonstrating that the function $w(\cdot)+\iota(\cdot)$
increases sufficiently quickly.
But $w(\cdot)+\iota(\cdot)$ actually increases quite slowly at times
when the workload fluid model
solution is near $d_{\max}$, i.e., possibly at small times. So one must
wait a short amount
of time (until time $\delta$ in the proof) before implementing this
strategy during which just a small amount of mass enters the entire system.
Once that small amount of time has elapsed, \eqref{eq:flwp} can be used
to obtain the desired bound.
The mathematical details are given next.

\begin{proof}[Proof of Lemma~\ref{prop:NearBoundary}] Fix $T,
\varepsilon>0$.
Set $\lambda_+=\sum_{k=1}^K\lambda_k$, $\delta=\varepsilon
/(4\lambda_+)$,
$M=w_u\vee w(\delta)$, and $c=\sum_{k=1}^K\rho_k G_k(M)$. Note that
$c>0$ since $M<d_{\max}$.
Further, by monotonicity properties of workload fluid model solutions,
$w(u)\le M$ for all $u\ge\delta$.

For $t\in[0, T]$, $1\le k\le K$, $x, y\in\Rp$, and $\kappa>0$, by
\eqref
{fdevoeq2},
\[
\zeta_k(t)\left(C_{(x, y)}^{\kappa}\right)
=
\zeta_k(0)\left(\left(C_{(x, y)}^{\kappa}\right)_t\right)
+
\lambda_k\int_0^t \left(\delta_{w(s)}^+\times\Gamma_k\right
)\left(\left
(C_{(x, y)}^{\kappa}\right)_{t-s}\right)\der s.
\]
By \eqref{eq:InitialCorner} there exists $\kappa_0$ such that for all
$0<\kappa<\kappa_0$,
\[
\max_{1\le k\le K} \sup_{t\in[0, T]} \sup_{x, y\in\Rp}\zeta
_k(0)\left
(\left(C_{(x, y)}^{\kappa}\right)_t\right)
<\frac{\varepsilon}{4}.
\]
Hence, for $t\in[0, T]$, $1\le k\le K$, $x, y\in\Rp$, and $0<\kappa
<\kappa_0$,
%
%
\begin{equation}\label{eq:fdevoeqk}
\zeta_k(t)\left(C_{(x, y)}^{\kappa}\right)
<
\frac{\varepsilon}{4}
+
\lambda_k\int_0^t \left(\delta_{w(s)}^+\times\Gamma_k\right
)\left(\left
(C_{(x, y)}^{\kappa}\right)_{t-s}\right)\der s.
\end{equation}
We must show that there exists $\kappa^*\le\kappa_0$ such that for all
$t\in[0, T]$, $1\le k\le K$, $x, y\in\Rp$, and $0<\kappa<\kappa^*$,
the integral term is bounded above by $3\varepsilon/4$.

Fix $t\in[0, T]$, $1\le k\le K$, and $x, y\in\Rp$.
For $s\in[0, T]$ and $\kappa>0$, let
\[
h_k(s, \kappa)=G_k((y-\kappa)^++t-s)-G_k(y+\kappa+t-s).
\]
For $0<\kappa<\kappa_0$, the integrand in \eqref{eq:fdevoeqk}
at time $s\in[0, t]$ is bounded above by
\[
\begin{cases}
0, &\hbox{if } w(s)+s<(x-\kappa)^++t, \\
1, &\hbox{if } (x-\kappa)^+ + t\le w(s)+s \le x+\kappa+t, \\
h_k(s, \kappa), &\hbox{if }x+\kappa+t<w(s)+s.
\end{cases}
\]
Equivalently, for $0<\kappa<\kappa_0$,
using continuity and monotonicity properties of $w(\cdot)+\iota(\cdot)$,
the integrand in \eqref{eq:fdevoeqk} at time $s\in[0, t]$ is bounded
above by
%
%
\begin{equation}\label{eq:bnds}
\begin{cases}
0, &\hbox{if } 0\le s<\tau((x-\kappa)^++t), \\
1, &\hbox{if } \tau((x-\kappa)^++t)\le s \le\tau(x+\kappa+t), \\
h_k(s, \kappa), &\hbox{if }\tau(x+\kappa+t)<s\le t.
\end{cases}
\end{equation}
This together with \eqref{eq:fdevoeqk} yields that for $0<\kappa
<\kappa_0$,
%
%
\begin{eqnarray}
\zeta_k(t)\left(C_{(x, y)}^{\kappa}\right)
&<&
\frac{\varepsilon}{4}
+
\lambda_k\int_{\tau((x-\kappa)^++t)\wedge t}^{\tau(x+\kappa
+t)\wedge
t}\der s
+\lambda_k\int_{\tau(x+\kappa+t)\wedge t}^t h_k(s, \kappa)\der
s\nonumber\\
&\le&
\frac{\varepsilon}{4}
+
\lambda_k\left( \tau(x+\kappa+t)\wedge t-\tau((x-\kappa
)^++t)\wedge t
\right)\nonumber\\
&&+\
\lambda_k \int_{(y-\kappa)^+}^{y+\kappa} G_k(u) \der u\nonumber\\
&\le&
\frac{\varepsilon}{4}
+
\lambda_k\left( \tau(x+\kappa+t)\wedge t-\tau((x-\kappa
)^++t)\wedge t
\right)
+
2\lambda_k\kappa.\label{eq:taumin}
\end{eqnarray}

For $0<\kappa<\kappa_0$, let
\[
\Delta(\kappa)= \tau(x+\kappa+t)\wedge t-\tau((x-\kappa
)^++t)\wedge t.
\]
Fix $0<\kappa<\kappa_0$.
If $\tau((x-\kappa)^++t)\ge t$, $\Delta(\kappa)=0$.
Also, if $x+\kappa+t\le w(0)$, then $\Delta(\kappa)=0$.
Henceforth, we assume that $\tau((x-\kappa)^++t)<t$ and $x+\kappa+t>w(0)$.
If $ \tau(x+\kappa+t)\wedge t\le\delta$, $\Delta(\kappa)\le
\delta$.
Otherwise, $ \tau(x+\kappa+t)\wedge t> \delta$, and there are two cases
to consider.
First consider the case where $\tau((x-\kappa)^++t)\ge\delta$.
Then $(x-\kappa)^++t> w(0)$ since $\tau((x-\kappa)^++t)>0$.
Hence, by \eqref{eq:tauprop} and \eqref{eq:flwp}, since $\tau
((x-\kappa
)^++t)\ge\delta$,
\begin{eqnarray*}
2\kappa
&=& x+\kappa+t-(x-\kappa+t)\\
&\ge& x+\kappa+t-((x-\kappa)^++t)\\
&=&w(\tau(x+\kappa+t))+\tau(x+\kappa+t)\\
&&- w(\tau((x-\kappa)^++t))-\tau((x-\kappa)^++t)\\
&=& \sum_{k=1}^{K} \rho_k\int_{\tau((x-\kappa)^++t)}^{\tau
(x+\kappa+t)}
G_k(w(u))\der u\\
&\ge& \sum_{k=1}^{K} \rho_k\int_{\tau((x-\kappa)^++t)}^{\tau
(x+\kappa
+t)} G_k(M)\der u\\
&\ge& c \Delta(\kappa).
\end{eqnarray*}
In particular, $\Delta(\kappa)\le2\kappa/c$.
The other case to consider is $\tau((x-\kappa)^++t)<\delta$. Then
\[
\Delta(\kappa)
\le\tau(x+\kappa+t)\wedge t
=\left(\tau(x+\kappa+t)\wedge t-\delta\right)+\delta.
\]
Further, by \eqref{eq:tauprop} and monotoncity of $w(\cdot)+\iota
(\cdot)$,
\begin{eqnarray*}
x-\kappa+t
&\le& (x-\kappa)^++t
\le w(\tau((x-\kappa)^++t))+\tau((x-\kappa)^++t)\\
&\le& w(\delta)+\delta
\le w(\tau(x+\kappa+t))+\tau(x+\kappa+t)
= x+\kappa+ t.
\end{eqnarray*}
Then, by \eqref{eq:tauprop} and \eqref{eq:flwp},
\begin{eqnarray*}
2\kappa
&\ge& x+\kappa+t- w(\delta)-\delta\\
&=&w(\tau(x+\kappa+t))+\tau(x+\kappa+t)- w(\delta)-\delta\\
&=& \sum_{k=1}^{K} \rho_k\int_{\delta}^{\tau(x+\kappa+t)}
G_k(w(u))\der
u\\
&\ge& \sum_{k=1}^{K} \rho_k\int_{\delta}^{\tau(x+\kappa+t)}
G_k(M)\der
u\\
&=& c\left(\tau(x+\kappa+t)-\delta\right).
\end{eqnarray*}
In particular,
%
%
\begin{equation}\label{eq:Delta}
\Delta(\kappa)\le\frac{2\kappa}{c}+ \delta,
\end{equation}
which is the largest of the four upper bounds.

By combining \eqref{eq:taumin}, \eqref{eq:Delta}, and the definition of
$\delta$
it follows that for $0<\kappa<\kappa_0$,
\begin{eqnarray*}
\zeta_k(t)\left(C_{(x, y)}^{\kappa}\right)
&<&
\frac{\varepsilon}{4}
+
\frac{2\lambda_k\kappa}{c}
+
\lambda_k\delta
+
2\lambda_k\kappa
\le
\frac{\varepsilon}{2}
+
\frac{2\lambda_k\kappa}{c}
+
2\lambda_k\kappa.
\end{eqnarray*}
Let $0<\kappa^*\le\kappa_0$ be such that for all $0<\kappa<\kappa^*$
\[
\max_{1\le k\le K}2\lambda_k\left(\frac{1}{c}+1\right)\kappa
<\frac
{\varepsilon}{2}.
\]
Then, for all $0<\kappa<\kappa^*$,
$\zeta_k(t)(C_{(x, y)}^{\kappa})<\varepsilon$.
Since $1\le k\le K$, $t\in[0, T]$, and $x, y, \in\Rp$ were chosen
arbitrarily, the result follows.
\end{proof}

Next, Lemma~\ref{prop:NearBoundary} is used to prove continuity for
Theorem~\ref{thrm:eu}.
For this, we need to specify a metric on $\M_2$ that induces the
topology of weak convergence.
For efficiency, we will specify such a metric on $\M_i$ for $i=1, 2$.
Given $i=1, 2$ and
$\zeta$, $\zeta'\in\M_i$, let ${\bf d}[\zeta, \zeta']$ denote the Prohorov
distance between
$\zeta$ and $\zeta'$. Specifically,
\begin{eqnarray*}
{\bf d}[\zeta, \zeta']
&=&
\inf\{ \varepsilon>0 : \zeta(B)\le\zeta'(B^{\varepsilon})
+\varepsilon
\hbox{ and }\zeta'(B)\le\zeta(B^{\varepsilon}) +\varepsilon\\
&&\qquad\hbox{ for all closed }B\in\CB_i\}.
\end{eqnarray*}
Then the following is a natural metric on $\M_i^K$. Given $i=1, 2$ and
$\zeta, \zeta'\in\M_i^K$,
define
%
%
\begin{equation}\label{eq:metric}
{\bf d}_K[\zeta, \zeta'] =\max_{1\le k\le K} {\bf d}[\zeta_k, \zeta'_k].
\end{equation}
In the following proof, we show that for each fluid model solution
$\zeta(\cdot)$,
$\lim_{h\to0} {\bf d}_K[\zeta(t+h), \zeta(t)]=0$ for all $t\in
\ptime$.

\begin{proof}[Proof of Property~3 for Theorem~\ref{thrm:eu}]
Let $\vartheta\in\I$ and let $\zeta(\cdot)$ be the unique fluid model
solution such that
$\zeta(0)=\vartheta$. It suffices to prove continuity of $\zeta
_k(\cdot
)$ on $\ptime$ for
each $1\le k\le K$. Fix $1\le k\le K$. We prove continuity of $\zeta
_k(\cdot)$
on $[0, T]$ for each $T>0$. Fix $T, \varepsilon>0$. By Lemma \ref
{prop:NearBoundary}, there exists
$\kappa$ such that for all $0<h<\kappa$,
\[
\sup_{s\in[0, T]}\zeta_k(s)(C^h) <\varepsilon.
\]
Let $0\le s < t \le T$ be such that $t-s<\min(\kappa, \varepsilon
/\max
(2, \lambda_k))$.
Note that $( \delta_{w(u)}^+\times\Gamma_k)(\Rp^2)=1$
for all $u\in\ptime$. Let $B\in\CB_2$ be closed and set $h=t-s>0$.
Note that $B_h\subset B^{2h}$.
By subtracting \eqref{fdevoeq2} applied to $B_h$ at time $s$ from
\eqref{fdevoeq2}
applied to $B$ at time $t$, we obtain
\begin{eqnarray*}
\zeta_k(t)(B)
&=&
\zeta_k(s)(B_h)+\lambda_k\int_s^t\left( \delta_{w(u)}^+\times
\Gamma
_k\right)(B_{t-u})\der u\\
&\le&
\zeta_k(s)(B^{2h})+\lambda_k h\\
& < &
\zeta_k(s)(B^{\varepsilon})+\varepsilon.
\end{eqnarray*}
Consider \eqref{fdevoeq2} applied to $B$ at time $s$, to $B^{2h}$ at
time $t$, and to $C^h$ at time $s$. For $u\in[0, s]$, we see that
$(w, p)\in B_u$ implies that
$(w-u, p-u)\in B$, which implies that either $(w-u, p-u)\in B\cap
C^h\subset C^h$
or $(w-u, p-u)\in B\setminus C^h$. If $(w-u, p-u)\in C^h$, then $(w,
p)\in
(C^h)_u$.
If $(w-u, p-u)\in B\setminus C^h$,
then $(w-u-h, p-u-h)\in B^{2h}$ and so $(w, p)\in(B^{2h})_{u+h}$
and $u+h=u+t-s\le t$. Therefore, for $u\in[0, s]$, $B_{s-u}\subset
(C^h)_{s-u}\cup
(B^{2h})_{t-v}$, where $v=t-(s-u+h)\in[0, t]$. Hence,
\begin{eqnarray*}
\zeta_k(s)(B)\le\zeta_k(t)(B^{2h})+ \zeta_k(s)(C^h)
<
\zeta_k(t)(B^{\varepsilon})+ \varepsilon.
\end{eqnarray*}
Then for all $0\le s<t\le T$ such that $0<t-s<\min(\kappa,
\varepsilon
/\max(2, \lambda_k))$,
\[
{\bf d}[\zeta_k(t), \zeta_k(s)]<\varepsilon.
\]
Hence, $\zeta_k(\cdot)$ is continuous on $[0, T]$. Since $T>0$ was
arbitrary, $\zeta_k(\cdot)$
is continuous on $\ptime$. Since $1\le k\le K$ was arbitrary, $\zeta
(\cdot)$ is continuous on $\ptime$.
\end{proof}

Now that each statement in Theorem~\ref{thrm:eu} has been verified, we
prove Theorem~\ref{thrm:is}.

\begin{proof}[Proof of Theorem~\ref{thrm:is}]
First suppose that $\vartheta\in\I$ is an invariant state. We must show
that $\vartheta\in\J$, i.e.,
we must show that for some $w\in[w_l, w_u]$, $\vartheta_k(B)=\theta
_k^w(B)$ for all
$B\in\CB_2$ and $1\le k\le K$.
Since $\vartheta$ is an invariant state, Theorem~\ref{thrm:eu} Property
1 implies that the unique workload
fluid model solution such that $w(0)=w_{\vartheta}$ is constant, i.e.,
$w(t)=w_{\vartheta}$ for all $t\in\ptime$.
Then, by the monotonicity properties of workload fluid model solutions,
$w_l\le w_{\vartheta}\le w_u$.
Let $w=w_{\vartheta}$. Fix $1\le k\le K$, $B\in{\mathcal P}$,
and $t>w$. Then $B=[a, \infty)\times[c, \infty)$ for some $a, c\in
\Rp$.
Further, $a+t-w>0$ and $(w-a)^+<t$. Hence, by \eqref{fdevoeq2}, we have
\begin{eqnarray*}
\vartheta_k(B)
&=& \lambda_k\int_0^t\left(\delta_w^+\times\Gamma_k\right
)(B_{t-s})\der
s\\
&=& \lambda_k\int_{(a+t-w)\wedge t}^t G_k(c+t-s)\der s\\
&=& \lambda_k\int_0^{(w-a)^+} G_k(c+u)\der u\\
&=& \theta_k^w(B).
\end{eqnarray*}
Since $B\in\CP$ was arbitrary, $\vartheta_k$ and $\theta_k^w$ agree on
$\CP$.
Since $\vartheta_k$ and $\theta_k^w$ agree on ${\mathcal P}$, they
agree on ${\mathcal R}$
and therefore they agree on $\CB_2$. So $\vartheta_k=\theta_k^w$. Since
$1\le k\le K$ was arbirary, $\vartheta=\theta^w\in\J$.

Next, fix $w_l\le w\le w_u$. Then the unique fluid workload solution
$w(\cdot)$ with $w(0)=w$
is constant, i.e., $w(t)=w$ for all $t\in\ptime$. Set $\zeta
(t)=\theta
^w$ for all $t\in\ptime$.
Then, in order to show that $\zeta(\cdot)$ is a fluid model solution,
it suffices to verify that
$\zeta(\cdot)$ satisfies \eqref{fdevoeq2}. For this, it suffices to
verify that
for all $B\in{\mathcal P}$, $1\le k\le K$, and $t>0$,
\[
\theta_k^w(B) = \theta_k^w(B_t)+\lambda_k\int_0^t\left(\delta
_w^+\times
\Gamma_k\right)(B_{t-s})\der s.
\]
For this fix $B\in{\mathcal P}$, $1\le k\le K$, and $t>0$. We have
\begin{eqnarray*}
\theta_k^w(B_t)&+&\lambda_k\int_0^t\left(\delta_w^+\times\Gamma
_k\right
)(B_{t-s})\der s\\
&=&
\theta_k^w(B_t)+\lambda_k\int_ {(a+t-w)^+\wedge t}^t G_k(c+t-s)\der s.
\end{eqnarray*}

\noindent\textit{Case~1}: Suppose that $a\ge w$.
Then $(a+t-w)^+\wedge t= t$ and it follows that
\[
\theta_k^w(B_t)+\lambda_k\int_0^t\left(\delta_w^+\times\Gamma
_k\right
)(B_{t-s})\der s
=0
=\theta_k^w(B).
\]

\noindent\textit{Case~2}:
Suppose that $a<w\le a+t$.
Then $(a+t-w)^+\wedge t=a+t-w$ and
\begin{eqnarray*}
\theta_k^w(B_t)&+&\lambda_k\int_0^t\left(\delta_w^+\times\Gamma
_k\right
)(B_{t-s})\der s\\
&=&
0+\lambda_k \int_ {a+t-w} ^ t G_k(c+t-s) \der s\\
&=&
\lambda_k \int_ {0} ^ {w-a}G_k(c+u)\der u\\
&=&
\theta_k^w(B).
\end{eqnarray*}

\noindent\textit{Case~3}:
If $a+t<w$, then $(a+t-w)^+\wedge t=0$ and
\begin{eqnarray*}
\theta_k^w(B_t)&+&\lambda_k\int_0^t\left(\delta_w^+\times\Gamma
_k\right
)(B_{t-s})\der s\\
&=&
\theta_k^w([a+t, w)\times[c+t, \infty))
+
\lambda_k \int_ 0^t G_k(c+t-s)\der s\\
&=&
\lambda_k\int_0^{w-a-t}G_k(c+t+u)\der u
+
\lambda_k \int_ 0^t G_k(c+u)\der u\\
&=&
\lambda_k\int_t^{w-a}G_k(c+u)\der u
+
\lambda_k \int_ 0^t G_k(c+u)\der u\\
&=&
\lambda_k \int_ 0^{w-a} G_k(c+u)\der u\\
&=&
\theta_k^w(B).
\end{eqnarray*}
Hence $\theta^w$ is an invariant state.
\end{proof}

\section{Proof of Tightness} \label{sec:tightness}

In this section, we prove that the sequence $\{\bzn(\cdot)\}_{n\in\N}$
is relatively
compact under the standing assumption \eqref{eq:IC}.
For this, we apply \cite[Corollary~3.7.4]{ref:EK}. In particular, it
suffices to
prove compact containment and an oscillation inequality. This is done in
Lemmas~\ref{lem:compact containment} and~\ref{lem:osc} below. Throughout,
\[
\lambda_+=\sum_{k=1}^K\lambda_k\qquad\hbox{and}\qquad g_+=\sum
_{k=1}^K\frac{1}{\gamma_k}.
\]

\subsection{Preliminaries}

For $1\le k\le K$ and $t\in\ptime$, let $\lambda_k^*(t)=\lambda_kt$.
For $1\le k\le K$, let $\dead_k^*(\cdot)=\lambda_k^* (\cdot) \Gamma_k$.
Define $\lambda^*(\cdot)=(\lambda_1^*(\cdot), \dots, \lambda
_K^*(\cdot))$
and $\dead^*(\cdot)=(\dead_1^*(\cdot)\dots, \dead_K^*(\cdot))$.
By a functional law of large numbers for renewal processes, as $n\to
\infty$,
%
%
\begin{equation}\label{eq:EFLLN}
\Ebn(\cdot)\Rightarrow\lambda^*(\cdot).
\end{equation}
Further, as a special case of Lemma~\ref{lem:deadFLLN} stated below,
the following
functional law of large numbers for the deadline process holds:
as $n\to\infty$,
%
%
\begin{equation}\label{eq:deadFLLN}
\left(\CDbn(\cdot), \left\langle \chi, \CDbn(\cdot)\right\rangle \right
)\Rightarrow
\left(\dead
^*(\cdot), \left\langle \chi, \dead^*(\cdot)\right\rangle \right).
\end{equation}
To state Lemma~\ref{lem:deadFLLN}, we need to introduce some
additional notation
and asymptotic assumptions.
\paragraph{Asymptotic Assumptions (AA)}
For each $n\in\N$ suppose that we have $K$ independent sequences of
strictly positive
independent and identically distributed random variables that are
independent of the
exogenous arrival process. For each $n\in\N$, denote the $k$th
sequence by
$\{g_{k, i}^n\}_{i\in\N}$, and for each $n\in\N$ and $1\le k\le K$,
denote the distribution of
$g_{k, 1}^n$ by $\Gamma_k^n$. We assume that $\Gamma_k^n$ has a
finite mean
$1/\gamma_k^n$, but we do not necessarily assume that $\Gamma_k^n$ is
continuous.
We further assume that for each $1\le k\le K$,
%
%
\begin{equation}\label{eq:ui}
\lim_{M\to\infty} \sup_{n\in\N} \left\langle \chi1_{(M, \infty)},
\Gamma
_k^n\right\rangle =0,
\end{equation}
and, as $n\to\infty$,
%
%
\begin{equation}\label{eq:weak}
\Gamma_k^n\wk\Gamma_k.
\end{equation}

One interpretation of (AA) is that for large $n$ and $1\le k\le K$,
$\Gamma_k^n$ approximates $\Gamma_k$.
So, for $n\in\N$ and $1\le k\le K$, one can regard $\{g_{k, i}^n\}
_{i\in
\N}$
as a collection of approximate patience times.
Here we are primarily interested in two particular choices of
approximate patience times. These are introduced below, shortly after
the statement of Lemma~\ref{lem:ResDeadFLLN}.

Note that \eqref{eq:ui} is a uniform integrability condition and,
together with \eqref{eq:weak},
it implies that for each $1\le k\le K$, as $n\to\infty$,
\[
1/\gamma_k^n\to1/\gamma_k.
\]
For each $n\in\N$, let $\Gamma^n=\Gamma_1^n\times\dots\times
\Gamma_K^n$
and let
$\Gamma=\Gamma_1\times\dots\times\Gamma_K$. For $n\in\N$, $1\le
k\le
K$, and
$t\in\ptime$, define
\[
{\mathcal G}_k^n(t)=\sum_{i=1}^{E_k^n(t)} \delta_{g_{k, i}^n}^+
\qquad\hbox{and}\quad
\bar{\mathcal G}_k^n(t)=\frac{1}{n}{\mathcal G}_k^n(t).
\]
For $n\in\N$, let ${\mathcal G}^n(\cdot)=({\mathcal
G}_1^n(\cdot
), \dots, {\mathcal G}_K^n(\cdot))$
and $\bar{\mathcal G}^n(\cdot)=(\bar{\mathcal G}_1^n(\cdot
), \dots,
\bar{\mathcal G}_K^n(\cdot))$.
So then, for each $n\in\N$, ${\mathcal G}^n(\cdot)\in\bD
(\ptime, \M_1^K)$. We refer to
${\mathcal G}^n(\cdot)$, $n\in\N$, as a deadline related process.

\begin{lemma}\label{lem:deadFLLN} Suppose that (AA) holds. Then, as
$n\to\infty$,
\[
\left(\bar{\mathcal G}^n(\cdot), \left\langle \chi, \bar{\mathcal
G}^n(\cdot
)\right\rangle
\right)
\Rightarrow
\left( \CD^*(\cdot), \left\langle \chi, \CD^*(\cdot)\right\rangle \right).
\]
\end{lemma}
Lemma~\ref{lem:deadFLLN} holds by \eqref{eq:EFLLN} and \cite[Theorem
5.1]{ref:GW}. To see this, note the following.
Respectively, the number of classes and deadline distributions play the
same role for the deadline related processes defined
here as the number of routes and service time distributions do for the
load process defined in~\cite{ref:GW}.
Then \cite[Theorem 5.1]{ref:GW} holds under \cite[Assumption
(A)]{ref:GW}, and the conditions in \cite[Assumption (A)]{ref:GW}
relevant to \cite[Theorem 5.1]{ref:GW} are \cite[(4.8)--(4.13)]{ref:GW}.
Conditions \cite[(4.8)--(4.9)]{ref:GW} hold here because the
deadlines are assumed to be strictly positive and have finite means.
Condition \cite[(4.10)]{ref:GW} corresponds to \eqref{eq:EFLLN}
above.
Conditions \cite[(4.11)--(4.13))]{ref:GW} hold here due to \eqref{eq:ui}
and \eqref{eq:weak} above.

For each $n\in\N$ and $1 \le k \le K$, define the fluid scaled
increments as follows:
for $0\le s\le t<\infty$,
\[
\Ebn_k(s, t) = \Ebn_k(t) - \Ebn_k(s)
\qquad\hbox{and}\qquad
\lambda^*_k(s, t) = \lambda^*_k(t)-\lambda^*_k(s),
\]
and
\[
\CGbn_k(s, t) = \CGbn_k(t) - \CGbn_k(s)
\qquad\hbox{and}\qquad
\dead^*_k(s, t) = \dead^*_k(t)-\dead^*_k(s).
\]
Then, by \eqref{eq:EFLLN}, given $T, \varepsilon, \eta>0$,
%
%
\begin{equation}\label{eq:EincFLLN}
\liminf_{n\to\infty} \bP
\left(
\max_{1\le k \le K} \sup_{0\le s \le t \le T}
\left| \Ebn_k(s, t) - \lambda^*_k(s, t) \right| \le\varepsilon
\right) \ge1-\eta.
\end{equation}
Similarly, the following corollary holds as a consequence of
Lemma~\ref{lem:deadFLLN} and continuity of the measures
$\Gamma_k$, $1 \le k \le K$.
\begin{cor}\label{cor:deadlineLimit}
Let $T, \varepsilon, \eta> 0$ and $0\le a\le b< \infty$. Suppose that
(AA) holds. Then
\[
\liminf_{n\to\infty} \bP
\left(
\max_{1\le k \le K} \sup_{0\le s \le t \le T}
\left| \CGbn_k(s, t)((a, b)) - \dead^*_k(s, t)((a, b)) \right| \le
\varepsilon
\right) \ge1-\eta.
\]
\end{cor}

We also wish to consider a version of a deadline related process
involving residual approximate patience times.
We refer to such a process as a residual deadline related process.
To define these, let $\{g_{k, i}^n\}_{i\in\N}$, $1\le k\le K$, and
$n\in
\N$,
satisfy (AA). For each $n\in\N$, $1\le k\le K$ and $t\in\ptime$, let
\[
g_{k, i}^n(t)=\left( g_{k, i}^n-\left(t-t_{k, Z_k^n(0)+i}^n\right
)^+\right)^+,
\]
and set
\[
{\mathcal R}_k^n(t)=\sum_{i=1}^{E_k^n(t)} \delta_{g_{k, i}^n(t)}^+
\qquad\hbox{and}\qquad
\bar{\mathcal R}_k^n(t)=\frac{{\mathcal R}_k^n(t)}{n}.
\]
Then, for each $n\in\N$ and $1\le k\le K$, ${\mathcal R}_k^n(\cdot
)\in\bD(\ptime, \M_1)$.
For $n\in\N$, set ${\mathcal R}^n(\cdot)=({\mathcal R}_1^n(\cdot),
\dots, {\mathcal R}_K^n(\cdot))$
and $\bar{\mathcal R}^n(\cdot)=(\bar{\mathcal R}_1^n(\cdot), \dots
, \bar{\mathcal R}_K^n(\cdot))$.
For each $1\le k\le K$ and $t\in\ptime$, let ${\mathcal R}_k^*(t)\in
\M
_1$ be the measure that is
absolutely continuous with respect to Lebesgue measure with density
$\lambda_k(G_k(\cdot)-\break G_k(\cdot+t))$.
Set ${\mathcal R}^*(\cdot)=({\mathcal R}_1^*(\cdot), \dots
, {\mathcal R}_K^*(\cdot))$.

\begin{lemma}\label{lem:ResDeadFLLN} Suppose that (AA) holds. Then, as
$n\to\infty$,
\[
\bar{\mathcal R}^n(\cdot)\Rightarrow{\mathcal R}^*(\cdot).
\]
\end{lemma}

In order to prove Lemma~\ref{lem:ResDeadFLLN}, we must first prove that
$\{ \bar{\mathcal R}^n(\cdot)\}_{n\in\N}$ is tight. This is done using
the same
general approach as that outlined for proving tightness of $\{\bzn
(\cdot
)\}_{n\in\N}$.
To illustrate similarity of proof techniques, we consecutively execute each
step for proving tightness of both processes in Sections~\ref{sec:cc}
through~\ref{sec:osc}
below. Then in Section~\ref{sec:char},
we complete the proof of Lemma~\ref{lem:ResDeadFLLN} by uniquely characterizing
the limit points.

There are two specific choices of $\{g_{k, i}^n\}_{i\in\N}$, $1\le
k\le
K$ and $n\in\N$, that will
be of particular interest for proving tightness of $\{\bzn(\cdot)\}
_{n\in\N}$ and Theorem~\ref{thrm:flt}.

\noindent\textit{Special Case 1:} For $n\in\N$, $1\le k\le K$, and
$i\in\N
$, let $g_{k, i}^n=d_{k, i}$.
This choice clearly satisfies (AA). Hence, Lemma~\ref{lem:deadFLLN}
implies \eqref{eq:deadFLLN}.
Further, for $n\in\N$, $1\le k\le K$, $i\in\N$, and $t\in\ptime$,
let
%
%
\begin{equation}\label{eq:a}
a_{k, i}^n(t)=\left( d_{k, i} - \left( t - t_{k, Z_k^n(0)+i}^n\right
)^+\right)^+.
\end{equation}
Then, for $n\in\N$ and $t\in\ptime$, set
\[
\CA_k^n(t)=\sum_{i=1}^{E_k^n(t)} \delta_{a_{k, i}^n(t)}^+
\qquad\hbox{and}\qquad
\CAbn_k(t)=\frac{1}{n}\CA^n_k(t).
\]
By Lemma~\ref{lem:ResDeadFLLN}, as $n\to\infty$,
%
%
\begin{equation}\label{eq:A}
\CAbn(\cdot)\Rightarrow{\mathcal R}^*(\cdot).
\end{equation}

\noindent\textit{Special Case 2:} For $n\in\N$, $1\le k\le K$, and
$i\in\N
$, let $g_{k, i}^n=d_{k, i}+v_{k, i}^n$.
This choice satisfies (AA) as well. For $n\in\N$, $1\le k\le K$,
$i\in\N
$, and $t\in\ptime$,
let
%
%
\begin{equation}\label{eq:v}
v_{k, i}^n(t)=\left( d_{k, i}+v_{k, i}^n - \left( t - t_{k,
Z_k^n(0)+i}^n\right)^+\right)^+.
\end{equation}
Then, for $n\in\N$ and $t\in\ptime$, set
\[
\CV_k^n(t)=\sum_{i=1}^{E_k^n(t)} \delta_{v_{k, i}^n(t)}^+
\qquad\hbox{and}\qquad
\CVbn_k(t)=\frac{1}{n}\CV_k^n(t).
\]
By Lemma~\ref{lem:ResDeadFLLN}, as $n\to\infty$,
%
%
\begin{equation}\label{eq:V}
\CVbn(\cdot)\Rightarrow{\mathcal R}^*(\cdot).
\end{equation}
Even though the steps for proving tightness of $\{ \bar{\mathcal
R}^n(\cdot)\}_{n\in\N}$
and $\{\bzn(\cdot)\}_{n\in\N}$ are executed consecutively, it is worth
noting that
Lemma~\ref{lem:ResDeadFLLN} is actually used to prove tightness of
$\{\bzn(\cdot)\}_{n\in\N}$ through the application of \eqref{eq:A}
and~\eqref{eq:V}.

\subsection{Compact Containment}\label{sec:cc}
In this section, we demonstrate the compact containment properties
needed to prove
tightness of $\{\CRbn(\cdot)\}_{n\in\N}$ and of $\{\bzn(\cdot)\}
_{n\in\N}$.
In both cases, we will utilize \cite[Lemma 15.7.5]{ref:K}. The
application turns out to be
simpler for $\{\CRbn(\cdot)\}_{n\in\N}$ since these are measures in~$\M_1^K$.

\paragraph{Compact Containment in $\M_1^K$}
Let $K_m=[0, m]$ and let $K_m^c$ denote its complement.
Then $\{K_m\}_{m\in\N}$ forms a sequence of sets that increases to
$\Rp$.
By \cite[Lemma 15.7.5]{ref:K}, $\K'\subset\M_1^K$ is relatively compact
if and only if there exists a positive constant $\check z$
and a sequence of positive constants $\{b_m\}_{m\in\N}$ tending to zero
such that for each $\zeta\in\K'$
\[
\max_{1\le k\le K} \zeta_k(\Rp) \le\check z
\qquad\mbox{and}\qquad
\max_{1 \le k \le K} \zeta_k(K_m^c) \le b_m, \ \forall\ m \in\N.
\]
Note that for $1\le k\le K$ and $m\in\N$, $\zeta_k(K_m^c)\le\langle\chi, \zeta_k\rangle /m$.
Hence it suffices to show
that there exist positive constants $\check z$ and $\check w$
such that for each $\zeta\in\K'$
%
%
\begin{equation}\label{eq:Kcompact1}
\max_{1\le k\le K} \zeta_k(\Rp) \le\check z
\qquad\mbox{and}\qquad
\max_{1 \le k \le K} \left\langle \chi, \zeta_k\right\rangle \le\check w.
\end{equation}
We will verify \eqref{eq:Kcompact1} to prove Lemma~\ref{lem:cc1}.

\begin{lemma} \label{lem:cc1} Suppose that (AA) holds.
Let $T, \eta>0$.
There exists a compact set $\K\subset\M^K_1$ such that
\[
\liminf_{n \to\infty}
\bP(\CRbn(t) \in\K\mbox{ for all } t \in[0, T]) \ge1- \eta.
\]
%
\end{lemma}
\begin{proof} Fix $T, \eta>0$. For all $1\le k\le K$ and $t\in[0, T]$,
%
%
\begin{equation}\label{eq:UBs}
\CRbn_k(t)(\Rp)\le\bar E_k^n(T)
\qquad\hbox{and}\qquad
\left\langle \chi, \CRbn_k(t) \right\rangle \le\left\langle \chi, \CGbn_k(T) \right\rangle .
\end{equation}
Define
\[
\Omega_1^n=\left\{\max_{1\le k\le K} \bar E_k^n(T) \le2\lambda_+
T\right\}
\quad\hbox{and}\quad
\Omega_2^n=\left\{\max_{1\le k\le K} \left\langle \chi, \CGbn_k(T)
\right\rangle \le
2g_+T\right\}.
\]
Set $\Omega_0^n=\Omega_1^n\cap\Omega_2^n$. By \eqref{eq:EincFLLN} and
Lemma~\ref{lem:deadFLLN},
%
%
\begin{equation}\label{eq:HighProb}
\liminf_{n \to\infty}\bP\left( \Omega^n_0 \right)\ge1 - \eta.
\end{equation}
The result follows by combining \eqref{eq:Kcompact1}, \eqref{eq:UBs}
and \eqref{eq:HighProb}.
\end{proof}

\paragraph{Compact Containment in $\M_2^K$}
For each $m \in\N$, let $K_m = [0, m] \times[0, m]$ and let $K_m^c$
denote its complement.
Then $\{K_m\}_{m\in\N}$ forms a sequence of sets that increases to
$\Rp^2$.
By \cite[Lemma 15.7.5]{ref:K}, $\K'\subset\M_2^K$ is relatively compact
if and only if there exists a positive constant $\check z$
and a sequence of positive constants $\{b_m\}_{m\in\N}$ tending to zero
as $m$ tends to infinity
such that for each $\zeta\in\K'$
%
%
\begin{equation}\label{eq:Kcompact}
\max_{1\le k\le K} \zeta_k(\Rp^2) \le\check z
\qquad\mbox{and}\qquad
\max_{1 \le k \le K} \zeta_k(K_m^c) \le b_m, \ \forall\ m \in\N.
\end{equation}
We will verify \eqref{eq:Kcompact} to prove Lemma~\ref{lem:compact
containment}.

\begin{lemma} \label{lem:compact containment}
Let $T, \eta>0$.
There exists a compact set $\K\subset\M^K_2$ such that
\[
\liminf_{n \to\infty}
\bP(\bzn(t) \in\K\mbox{ for all } t \in[0, T]) \ge1- \eta.
\]
%
\end{lemma}
\begin{proof}
Fix $T, \eta> 0$.
First we identify a sequence of sets $\{\Omega_0^n\}_{n\in\N}$ for which
we will show that for each $n\in\N$ on $\Omega_0^n$, $\bzn(t)$ remains
in a particular relatively compact set for all $t\in[0, T]$.
By \eqref{eq:IC}, there exists a compact set $\K_0$ such that
\[
\liminf_{n\to\infty}\bP(\bzn(0)\in\K_0)\ge1-\frac{\eta}{4}.
\]
Since $\K_0$ is compact, there exists a positive constant $\check z_0$
and a sequence of positive constants $\{ a_m\}_{m\in\N}$
tending to zero as $m$ tends to infinity such that
\[
\K_0\subset\left\{ \zeta\in\M^K_2 :
\max_{1\le k\le K} \zeta_k(\Rp^2) \le\check z_0
\mbox{ and }
\max_{1 \le k \le K} \zeta_k(K_m^c) \le a_m, \ \forall\ m \in\N
\right\}.
\]
For $n\in\N$, let
\[
\Omega^n_1
=
\left\{
\max_{1\le k\le K} \bzn_k(0)(\Rp^2) \le\check z_0
\hbox{ and }
\max_{1 \le k \le K} \bzn_k(0)(K_m^c)\le a_m, \ \forall\ m \in\N
\right\}.
\]
Then
\[
\liminf_{n \to\infty}
\bP\left(\Omega^n_1 \right)
\ge1- \frac{\eta}{4}.
\]
%

For $n\in\N$, let
\[
\Omega^n_2
=
\left\{
\max_{1\le k\le K} \Ebn_k(T)
\le
2\lambda_+ T
\right\}.
\]
By \eqref{eq:EincFLLN},
\[
\liminf_{n \to\infty}
\bP\left( \Omega^n_2 \right)
\ge
1-\frac{\eta}{4}.
\]
Note that, for all $n\in\N$,
\[
\sup_{t\in[0, T]}W^n(t)\le W^n(0) + \frac{1}{n}\sum_{k=1}^K \sum
_{i=1}^{E_k^n(T)} v_{k, i}.
\]
Then, by \eqref{eq:IC} and a functional strong law of large numbers,
there exists a positive
constant $\check{w}$ such that
\[
\liminf_{n\to\infty} \bP\left( \sup_{t\in[0, T]} W^n(t) \le
\check
{w}\right)\ge1-\frac{\eta}{4}.
\]
For $n\in\N$, let
\[
\Omega^n_3
=
\left\{\sup_{t\in[0, T]} W^n(t) \le\check{w} \right\}.
\]
Notice that for each $n, m\in\N$, $1\le k\le K$, and $t \in[0, T]$,
%
%
\begin{equation}\label{eq:TailBnd}
\bar\dead^n_k(t)((m, \infty))
\le
\bar\dead^n_k(T)((m, \infty))
\le
\frac{\left\langle \chi, \bar\dead^n_k(T) \right\rangle }{m}.
\end{equation}
Set $c = 2 \lambda_+g_+ T$.
For each $m \in\N$, let $c_m=c/m$.
Then \eqref{eq:deadFLLN} and \eqref{eq:TailBnd} together imply that
\[
\liminf_{n \to\infty}
\bP\left(
\max_{1 \le k \le K}\sup_{t\in[0, T]}
\bar\dead^n_k(t)((m, \infty))
\le
c_m, \
\forall\ m \in\N
\right)
\ge
1-\frac{\eta}{4}.
\]
For $n\in\N$, let
\[
\Omega^n_4
=
\left\{
\max_{1 \le k \le K}\sup_{t\in[0, T]}
\bar\dead^n_k(t)((m, \infty))
\le
c_m, \ \forall\ m \in\N
\right\}.
\]
Finally, for each $n\in\N$ set
\[
\Omega^n_0 = \Omega^n_1 \cap\Omega^n_2 \cap\Omega^n_3 \cap\Omega^n_4.
\]
It follows that
%
%
\begin{equation} \label{eq:moreNiceOmegas}
\liminf_{n \to\infty}\bP\left( \Omega^n_0 \right)\ge1 - \eta.
\end{equation}

Next we identify the relatively compact set $\K'$.
For this, let $\check z=\check z_0 + 2\lambda_+T$ and for $m\in\N$, let
\[
b_m=
\begin{cases}
\check z, & 1\le m\le\check w, \\
a_m+c_{m-\check w}, & m>\check w.
\end{cases}
\]
Then $\{b_m\}_{m\in\N}$ is a sequence of postive numbers tending
to zero as $m$ tends to infinity. Let
\[
\K'=\left\{\zeta\in\M^K_2:
\max_{1 \le k \le K} \zeta_k(\Rp^2) \le\check z
\mbox{ and }
\max_{1 \le k \le K} \zeta_k(K_m^c) \le b_m, \forall m \in\N
\right\}.
\]
Then, by \eqref{eq:Kcompact}, $\K'$ is relatively compact.

It suffices to show that for each $n\in\N$, on $\Omega^n_0$,
$\bzn(t) \in\K'$ for all $t \in[0, T]$. Fix $n\in\N$.
On $\Omega^n_1 \cap\Omega^n_2$ we have that, for all $1 \le k \le K$
and $t \in[0, T]$,
%
%
\begin{equation}\label{eq:TM}
\bzn_k(t)(\Rp^2)
\le
\bzn_k(0)(\Rp^2) + \bar{E}^n_k(T)
\le
\check z_0 + 2\lambda_+T=\check z.
\end{equation}
Next, by \eqref{eq:dynamics1} under fluid scaling and the fact that for
any $x\in\Rp$ and $m\in\N$,
$(K_m^c)_x\subseteq K_m^c$, we have, for each $m\in\N$, $1 \le k \le
K$, and $t\in[0, T]$,
\begin{eqnarray*}
\bzn_k(t) (K^c_m)
&=&
\frac{1}{n}
\sum_{j = 1}^{A^n_k(t)}
1_{ \left(K^c_m\right)_{t-t_{k, j}^n} }(w_{k, j}^n, p_{k, j}^n) \\
&\le&
\frac{1}{n}
\sum_{j = 1}^{A^n_k(t)}
1_{K^c_m }(w_{k, j}^n, p_{k, j}^n) \\
&\le&
\bzn_k(0)(K_m^c)+
\frac{1}{n}
\sum_{j =Z^n_k(0)+1}^{A^n_k(t)}\left(
1_{\left\{w^n_{k, j} > m\right\}}
+
1_{\left\{p^n_{k, j} > m\right\}}\right).
\end{eqnarray*}
Recall that on $\Omega^n_1$, $\max_{1\le k\le K}\bzn_k(0)(K^c_m) \le
a_m$ for all $m\in\N$.
Further, on $\Omega^n_3$, for $m > \check w$, $1_{\{w^n_{k, j} >m\}}=0$
for each $1\le k\le K$ and $Z^n_k(0)+1\le j \le A^n_k(T)$.
Additionally, on $\Omega^n_3$,
for each $1\le k\le K$ and $Z^n_k(0)+1\le j \le A^n_k(T)$,
\[
p^n_{k, j}\le d_{k, j-Z^n_k(0)} +W^n(t_{k, j}^n)\le d_{k, j-Z^n_k(0)} +
\check{w}.
\]
Hence, on $\Omega^n_3 \cap\Omega^n_4$, for $m>\check w$ and $1\le
k\le K$,
\[
\frac{1}{n}
\sum_{j = Z^n_k(0)+1}^{A^n_k(t)}
1_{\left\{p^n_{k, j} > m\right\}}
\le
\bar\dead^n_k(t)((m-\check w, \infty))
\le
c_{m-\check w}.
\]
Then on $\Omega_1^n\cap\Omega_3^n\cap\Omega_4^n$ for $m>\check w$,
%
%
\begin{equation}\label{eq:TailM}
\bzn_k(t) (K^c_m)\le a_m+c_{m-\check w}.
\end{equation}
By \eqref{eq:TM} and \eqref{eq:TailM}, it follows that on $\Omega^n_0$,
$\bzn_k(t) (K^c_m)\le b_m$ for all $m\in\N$. This together with
\eqref{eq:TM}
implies that on $\Omega^n_0$, $\bzn(t) \in\K'$ for each $t \in[0, T]$.
\end{proof}

\subsection{Asymptotic Regularity}
This section contains results that are preparatory for proving the
oscillation bounds.
For the sequences of fluid scaled residual deadline related processes 
and state descriptors, 
the sudden arrival or departure of a
large amount of mass may result in a large oscillation. The focus here
is showing
that it is very unlikely that large
oscillations take place due to departing mass.

\paragraph{Asymptotic Regularity in $\M_1^K$}
Consider the sequence $\{ \CRbn(\cdot)\}_{n\in\N}$ of fluid scaled
residual deadline related processes.
Given $x\in\Rp$ and $\kappa>0$, let
\[
I_x^{\kappa}=\left( (x-\kappa)^+, x+\kappa\right).
\]
Note that for $t\in\ptime$, $x\in\Rp$, and $\kappa>0$, the mass in
$I_{x}^{\kappa}$ at time $t$ will all depart the system
during the time interval $(t+(x-\kappa)^+, t+x+\kappa)$. This time
interval is small if $\kappa$ is small.
Hence, in order to avoid an abrupt departure of a large amount of mass,
one needs to show that asymptotically
such sets contain arbitrarily small mass. This is stated precisely in
the following lemma.

\begin{lemma}\label{lem:BndReg1} Suppose that (AA) holds.
Let $T, \varepsilon, \eta>0$. Then there exists $\kappa>0$ such that
\[
\liminf_{n\to\infty}\bP\left(
\max_{1 \le k \le K}
\sup_{t\in[0, T]} \sup_{x\in\Rp} \CRbn_k(t)(I_x^{\kappa})\le
\varepsilon
\right)\ge1-\eta.
\]
\end{lemma}

\begin{proof}
Fix $T, \varepsilon, \eta> 0$.
Set $\kappa= \frac{\varepsilon}{8 \lambda_+}$ and $M_t = \lceil
t/\kappa\rceil$ for $0< t \le T$.
For each integer $m \ge3$, $n\in\N$, $1\le k\le K$, and $t \in[0, T]$,
\[
\CGbn_k(t)(((m-2) \kappa, \infty))
\le
\CGbn_k(T)(((m-2)\kappa, \infty))
\le
\frac{\left\langle \chi, \CGbn_k(T) \right\rangle }{(m-2) \kappa}.
\]
It follows by Lemma~\ref{lem:deadFLLN} that for some $m \ge3$,
\[
\liminf_{n \to\infty}
\bP\left(
\max_{1 \le k \le K}\sup_{t\in[0, T]}
\CGbn_k(t)(((m - 2) \kappa, \infty))
\le
\varepsilon
\right)
\ge
1-\frac{\eta}{2}.
\]
Fix such an $m$ and for $n\in\N$, let
\[
\Omega^n_1
=
\left\{
\max_{1 \le k \le K}\sup_{t\in[0, T]}
\CGbn_k(t)(((m- 2) \kappa, \infty))
\le
\varepsilon
\right\}.
\]
Also, for $n\in\N$, define
\begin{eqnarray*}
\Omega^n_2
&=&
\left\{
\max_{1 \le k \le K}
\sup_{0\le s\le t\le T}
\max_{0\le j \le m+M_T}
\left|
\CGbn_k(s, t)( (j\kappa, (j+4)\kappa))\right.\right.\\
&&\qquad\left. \left. -
\dead^*_k(s, t)( (j\kappa, (j+4)\kappa))
\right|
\le\frac{\varepsilon}{2 M_T}
\right\}.
\end{eqnarray*}
By Corollary~\ref{cor:deadlineLimit},
\[
\liminf_{n \to\infty}\bP\left(\Omega^n_2\right) \ge1-\frac
{\eta}{2}.
\]
For $n\in\N$, set $\Omega_0^n = \Omega^n_1 \cap\Omega^n_2$. Then, we
have that
\[
\liminf_{n \to\infty}\bP\left(\Omega_0^n \right) \ge1-\eta.
\]

Fix $n\in\N$, $1 \le k \le K, x\in\Rp, $ and $0 < t \le T$.
First, note that for any $1 \le i \le E^n_k(t)$, we have
$t^n_{k, i+Z^n_k(0)} \le t$,
so that $(t-t^n_{k, i+Z^n_k(0)})^+ = t-t^n_{k, i+Z^n_k(0)}$.
It follows that
\[
\CRbn_k(t)(I_x^{\kappa})
=
\frac{1}{n}
\sum_{i=1}^{E^n_k(t)}
1_{I^\kappa_x}(g^n_{k, i}(t)) \\
\le
\frac{1}{n}
\sum_{i=1}^{E^n_k(t)}
1_{I^\kappa_{x+t-t^n_{k, i+Z^n_k(0)}}}(g^n_{k, i}).
\]
Then, for $x\ge m\kappa$ and $1 \le i \le E^n_k(t)$,
\[
\left(x+t-t^n_{k, i+Z^n_k(0)}-\kappa\right)^+\ge x-\kappa\ge
(m-1)\kappa
>(m-2)\kappa.
\]
Hence, on $\Omega_1^n$, for $x\ge m\kappa$,
\[
\CRbn_k(t)(I_x^{\kappa})\le\frac{1}{n}
\sum_{i=1}^{E^n_k(t)}
1_{ ((m-2)\kappa, \infty) }(g^n_{k, i})=\CGbn_k(t)\left
(((m-2)\kappa
, \infty)\right)\le\varepsilon.
\]
Otherwise, $x< m\kappa$. Then
\begin{eqnarray*}
\CRbn_k(t)(I_x^{\kappa})
&\le&
\frac{1}{n}\left(
\sum_{j=0}^{M_t-2}
\sum_{i=E^n_k(j \kappa) +1}^{E^n_k((j+1)\kappa)}
1_{I^\kappa_{x+t-t^n_{k, i+Z^n_k(0)}}}(g^n_{k, i})\right.\\
&&+
\left.\sum_{i=E^n_k((M_t-1) \kappa) +1}^{E^n_k(t)}
1_{I^\kappa_{x+t-t^n_{k, i+Z^n_k(0)}}}(g^n_{k, i})\right).
\end{eqnarray*}
Noting that for $0 \le j \le M_t-1$ and $E^n_k(j\kappa) < i \le
E^n_k((j+1)\kappa\vee t)$ we have that
\[
x+t-\kappa- t^n_{k, i+Z^n_k(0)}
\ge
x+t - (j+2) \kappa
\ge\kappa\left( \left\lfloor\frac{x+t}{\kappa} \right\rfloor-
j -
2\right)
\]
and
\[
x+t +\kappa- t^n_{k, i+Z^n_k(0)}
\le
x+t -( j-1) \kappa
\le\kappa\left( \left\lfloor\frac{x+t}{\kappa} \right\rfloor-
j +
2\right).
\]
Then it follows that
\begin{eqnarray*}
\CRbn_k(t)(I_x^{\kappa})
&\le&
\frac{1}{n}
\sum_{j=0}^{M_t-2}
\sum_{i=E^n_k(j \kappa) +1}^{E^n_k((j+1)\kappa)}
1_{I^{2 \kappa}_{\kappa\left( \left\lfloor\frac{x+t}{\kappa}
\right
\rfloor-j \right)}}(g^n_{k, i})\\
&&+\frac{1}{n}
\sum_{i=E^n_k((M_t-1) \kappa) +1}^{E^n_k(t)}
1_{I^{2 \kappa}_{\kappa\left( \left\lfloor\frac{x+t}{\kappa}
\right
\rfloor-(M_t-1) \right)}}(g^n_{k, i}) \\
&=&
\sum_{j=0}^{M_t-2}
\bar{\mathcal{G}}^n_k(j \kappa, (j+1)\kappa)
\left(I^{2 \kappa}_{\kappa\left( \left\lfloor\frac{x+t}{\kappa}
\right\rfloor-j \right)}\right)\\
&&+ \bar{\mathcal{G}}^n_k((M_t-1)\kappa, t )
\left(I^{2 \kappa}_{\kappa\left( \left\lfloor\frac{x+t}{\kappa}
\right\rfloor-(M_t-1) \right)}\right).
\end{eqnarray*}
Since $x<m\kappa$, for all $1\le j\le M_t-1$, the left end point of
$I^{2 \kappa}_{\kappa(\lfloor\frac{x+t}{\kappa}
\rfloor-j)}$ satisfies
\[
0
\le
\left( \kappa\left( \left\lfloor\frac{x+t}{\kappa} \right
\rfloor-j
\right)-2\kappa\right)^+
\le
x+t < m\kappa+M_t\kappa\le(m+M_T)\kappa.
\]
Also note that $0<t-(M_t-1)\kappa\le t- (t/\kappa-1)\kappa=\kappa$.
Then, on $\Omega^n_2$,
since $x< m \kappa$,
\begin{eqnarray*}
\CRbn_k(t)(I_x^{\kappa})
&\le&
\sum_{j=0}^{M_t-2}
{\mathcal{D}}^*_k(j \kappa, (j+1)\kappa)
\left(I^{2 \kappa}_{\kappa\left( \left\lfloor\frac{x+t}{\kappa}
\right\rfloor-j \right)} \right)\\
&&+ {\mathcal{D}}^*_k((M_t-1)\kappa, t)
\left(I^{2 \kappa}_{\kappa\left( \left\lfloor\frac{x+t}{\kappa}
\right\rfloor-(M_t-1) \right)} \right)
+ \frac{\varepsilon}{2}\\
&\le&
\sum_{j=0}^{M_t-1}
\lambda_k \kappa\left[
F_k \left(\kappa\left( \left\lfloor\frac{x+t}{\kappa} \right
\rfloor
-j + 2 \right)^+\right) \right. \\
&&\left.\qquad\qquad\qquad-
F_k \left(\kappa\left( \left\lfloor\frac{x+t}{\kappa} \right
\rfloor
-j - 2 \right)^+\right)
\right] + \frac{\varepsilon}{2} \\
&\le&
\lambda_k \kappa
\sum_{j=-2}^\infty\left[
F_k((j+4)\kappa) - F_k(j\kappa)
\right]
+\frac{\varepsilon}{2}.
\end{eqnarray*}
Note that $F_k(\cdot)$ is a cumulative distribution function and that
each point in $\Rp$
is included in at most four intervals of the form $(j\kappa
, (j+4)\kappa
]$, $j=-2, -1, 0, 1, 2, \dots$. Thus,
on $\Omega^n_2$, since $x<m\kappa$,
\[
\CRbn_k(t)(I_x^{\kappa}) \le4 \lambda_k \kappa+ \frac{\varepsilon}{2}
\le\varepsilon.
\]
Since, $n\in\N$, $1\le k\le K$, $x\in\Rp$, and $t\in[0, T]$ were chosen
arbitrarily, this concludes the proof.
\end{proof}

\paragraph{Asymptotic Regularity in $\M_2^K$}
We need to prove an analog of Lemma~\ref{lem:BndReg1} for $\{\bzn
(\cdot
)\}_{n\in\N}$.
In particular, we wish to prove the following
prelimit version of Lemma~\ref{prop:NearBoundary}.

\begin{lemma}\label{lem:BndReg}
Let $T, \varepsilon, \eta>0$. Then there exists $\kappa>0$ such that
\[
\liminf_{n\to\infty}\bP\left(
\max_{1 \le k \le K}
\sup_{x\in\Rp^2} \sup_{t\in[0, T]} \bzn_k(t)(C_x^{\kappa})\le
\varepsilon
\right)\ge1-\eta.
\]
\end{lemma}

Before proving Lemma~\ref{lem:BndReg}, we verify the following
regularity result for the initial state,
which is the stochastic analog of \eqref{eq:InitialCorner}.

\begin{lemma}\label{lem:BndReg0} Let $\varepsilon, \eta>0$. Then there
exists $\kappa>0$ such that
\[
\liminf_{n\to\infty}\bP\left(
\max_{1\le k\le K}
\sup_{x\in\Rp^2} \bzn_k(0)(C_x^{\kappa})\le\varepsilon\right
)\ge1-\eta.
\]
\end{lemma}

\begin{proof} Fix $\varepsilon, \eta>0$. Given $i=1, 2$, recall
definition \eqref{def:pi} of the projection
mapping $\pi_i:\M_2\to\M_1$. We apply the argument given in
\cite[Pages 835--836]{ref:GPW} for measures in $\M_1$ to the
projection mappings
applied to $\bzn_+(0)$, $n\in\N$, to verify that there exists
$\kappa>0$
such that
%
%
\begin{equation}\label{eq:d1}
\liminf_{n\to\infty}
\bP\left( \max_{i=1, 2}\sup_{x\in\Rp} \left\langle 1_{[(x-\kappa)^+,
x+\kappa]},
\pi_i\left({\bzn_+(0)}\right)\right\rangle
<
\frac{\varepsilon}{2}\right)\ge1-\eta.
\end{equation}
The desired result follows from \eqref{eq:d1} since for all $n\in\N$,
$1\le k\le K$, $x=(x_1, x_2)\in\Rp^2$ and
$\kappa>0$,
\[
\bzn_k(0)(C_x^{\kappa})
\le
\left\langle 1_{[(x_1-\kappa)^+, x_1+\kappa]}, \pi_1\left(\bzn
_+(0)\right
)\right\rangle
+
\left\langle 1_{[(x_2-\kappa)^+, x_2+\kappa]}, \pi_2\left(\bzn
_+(0)\right
)\right\rangle .
\]

In order to verify \eqref{eq:d1}, we must verify that suitable
$K$-dimensional analogs of
\cite[(3.19)--(3.22)]{ref:GPW} hold. For this, for $\zeta\in\M_2^K$
and $i=1, 2$ we adopt
the shorthand notation
\[
\pi_i(\zeta)=\left( \pi_i(\zeta_1), \dots, \pi_i(\zeta_K)\right
)\hbox{
and }
\left\langle \chi, \pi_i(\zeta)\right\rangle = \left(\left\langle \chi, \pi_i(\zeta
_1)\right\rangle ,
\dots, \left\langle \chi
, \pi_i(\zeta_K)\right\rangle
\right).
\]
Note that for $\nu\in\M_2$ and $i=1, 2$, $\langle \chi, \pi_i(\nu)
\rangle
=\langle \chi
\circ p_i, \nu\rangle =\langle p_i, \nu\rangle $.
Then, by \eqref{eq:IC}, as $n\to\infty$,
\begin{eqnarray*}
&&\left( \pi_1(\bzn(0)), \pi_2(\bzn(0)), \left\langle \chi, \pi_1(\bzn
(0))\right\rangle , \left\langle
\chi, \pi_2(\bzn(0))\right\rangle \right)\\
&&\qquad\Rightarrow
\left( \pi_1(\z_0^*), \pi_1(\z_0^*), \left\langle \chi, \pi_1(\z_0^*)
\right\rangle ,
\left\langle \chi
, \pi_1(\z_0^*) \right\rangle \right),
\end{eqnarray*}
which is the $K$-dimensional analog of \cite[(3.19)]{ref:GPW}. Using
(A.3) we obtain the following $K$-dimensional
analog of \cite[(3.20)]{ref:GPW}:
\[
\max_{1\le k\le K}\max_{i=1, 2} \bE\left[ \left\langle 1, \pi_i(\z
_{0, k}^*)\right\rangle
\right]
\le
\bE\left[ \z_{0, +}^*(\Rp^2) \right]<\infty.
\]
Using (A.2) we obtain the following $K$-dimensional analog of \cite
[(3.21)]{ref:GPW}:
\[
\max_{1\le k\le K}\max_{i=1, 2}\bE\left[\left\langle \chi, \pi_i(\z
_{0, k}^*)\right\rangle
\right]
\le
\bE\left[ \left\langle p_1+p_2, \z_{0, +}^*\right\rangle \right]<\infty.
\]
Using (A.1) (and in particular (I.1)) gives the following
$K$-dimensional analog of \cite[(3.22)]{ref:GPW}:
\[
\bP\left( \max_{1\le k\le K}\max_{i=1, 2}\sup_{x\in\Rp} \left
<1_{\{
x\}},
\pi_i(\z_{0, k}^*)\right\rangle =0\right)=1.\qedhere
\]
\end{proof}

Before moving on to prove Lemma~\ref{lem:BndReg}, we obtain an almost sure
upper bound on the mass in $C_{(x, y)}^\kappa$ in an arbitrary
coordinate of the $n$th
system at time $t$ for each $\kappa>0$, $x, y \in\Rp$, and $t\in
\ptime$.
To this end,
let $n\in\N$, $1\le k\le K$, $t\in\ptime$, $x, y \in\Rp$, and
$\kappa
>0$. The $n$th
system analog of \eqref{eq:dynamics2} for the set $C_{(x, y)}^\kappa$ is
\begin{eqnarray*}
\zn_k(t)(C_{(x, y)}^\kappa)
&\le&
\zn_k(0)\left( \left(C_{(x, y)}^{\kappa}\right)_t\right)
+
\sum_{j=Z^n_k(0)+1}^{A^n_k(t)}
1_{\left(C_{(x, y)}^{\kappa}\right)_{t-t_{k, j}^n}} (w_{k, j}^n,
p_{k, j}^n)
\\
&=&
\zn_k(0)\left( \left(C_{(x, y)}^{\kappa}\right)_t\right)
+
\sum_{j=Z^n_k(0)+1}^{A^n_k(t)}
1_{C_{(x, y)}^{\kappa}} (w_{k, j}^n(t), p_{k, j}^n(t)).
\end{eqnarray*}
But then, since the $(x, y)$-shift of a set followed by the $\kappa
$-enlargement
contains the $\kappa$-enlargement followed by the $(x, y)$-shift,
%
%
\begin{equation}\label{eq:jobsInCorner1}
\zn_k(t)(C_{(x, y)}^{\kappa})
\le
\zn_k(0)\left(C_{(x+t, y+t)}^{\kappa}\right)
+
\sum_{j=Z^n_k(0)+1}^{A^n_k(t)}
1_{ C_{(x, y)}^{\kappa} } (w_{k, j}^n(t), p_{k, j}^n(t)).
\end{equation}
We simplify the summation term by focusing on each coordinate separately.
For this, we classify jobs by those whose residual patience time causes the
associated unit atom to lie in a certain horizontal band and those
whose residual
virtual sojourn time causes the associated unit atom to lie in a certain
vertical band. Then, for each $n\in\N$, $1\le k\le K$, $t\in\ptime$,
$x, y \in\Rp$, and
$\kappa> 0$,
%
%
\begin{eqnarray} \label{eq:jobsInCorner}
\zn_k(t)(C_{(x, y)}^{\kappa})
&\le&
\zn_k(0)(C_{(x+t, y+t)}^{\kappa})
+
\sum_{j=Z^n_k(0)+1}^{A^n_k(t)}
1_{I^\kappa_{x}} (w^n_{k, j}(t))\\
&&+
\sum_{j=Z^n_k(0)+1}^{A^n_k(t)}
1_{I^\kappa_{y}} (p^n_{k, j}(t)).\nonumber
\end{eqnarray}
We are prepared to prove Lemma~\ref{lem:BndReg}. The proof given below
can be
regarded as a stochastic version of the proof of Lemma~\ref{prop:NearBoundary}.

\begin{proof}[Proof of Lemma~\ref{lem:BndReg}]
Fix $\varepsilon$, $\eta>0$ and $T>\frac{\varepsilon}{24 \lambda_+}$.
We begin by defining a sequence $\{\Omega_0^n\}_{n\in\N}$ of events
on which the prelimit processes satisfy properties analogous to those
exhibited by fluid model solutions and used in the proof
of Lemma~\ref{prop:NearBoundary}.
By Lemma~\ref{lem:BndReg0}, there exists $\kappa_0>0$ such that
\[
\liminf_{n\to\infty} \bP\left( \max_{1\le k\le K} \sup_{x, y \in
\Rp}
\bzn_k(0)\left(C_{(x, y)}^{\kappa_0}\right)\le\frac{\varepsilon
}{4}\right)\ge1-\frac{\eta}{6}.
\]
For $n\in\N$, let
\[
\Omega^n_1
=
\left\{\max_{1 \le k \le K}
\sup_{x, y \in\Rp} \bzn_k(0)\left(C_{(x, y)}^{\kappa_0}\right
)\le
\frac
{\varepsilon}{4}\right\}.
\]
By \eqref{eq:A} and \eqref{eq:V}, there exists a $\kappa_1 > 0$ such that
\[
\liminf_{n\to\infty} \bP\left( \max_{1\le k\le K} \sup_{y\in\Rp
} \sup
_{t \in[0, T]}
\bar{\mathcal{A}}^n_k(t)(I_y^{\kappa_1})\le\frac{\varepsilon
}{4}\right
)\ge1-\frac{\eta}{6}
\]
and
\[
\liminf_{n\to\infty} \bP\left( \max_{1\le k\le K} \sup_{y \in
\Rp} \sup
_{t \in[0, T]}
\bar{\mathcal{V}}^n_k(t)(I_y^{\kappa_1})\le\frac{\varepsilon
}{4}\right
)\ge1-\frac{\eta}{6}.
\]
For $n\in\N$, let
\[
\Omega^n_2
=
\left\{\max_{1 \le k \le K}
\sup_{y \in\Rp}
\sup_{t \in[0, T]}
\bar{\mathcal{A}}^n_k(t)(I_y^{\kappa_1})\le\frac{\varepsilon
}{4}\right
\}
\]
and
\[
\Omega^n_3
=
\left\{\max_{1 \le k \le K}
\sup_{y \in\Rp}
\sup_{t \in[0, T]}
\bar{\mathcal{V}}^n_k(t)(I_y^{\kappa_1})\le\frac{\varepsilon
}{4}\right
\}.
\]
For $n\in\N$, let
\begin{eqnarray*}
\Omega^n_4
&=&
\left\{\max_{1\le k\le K} \sup_{0\le s\le t\le T}
\bar{A}^n_k(t) - \bar{A}^n_k(s) \le2 \lambda_+ (t-s)
\right\}.
\end{eqnarray*}
Since for all $n\in\N$, $1\le k\le K$, and $0\le s\le t\le T$, $\bar
{A}^n_k(t) - \bar{A}^n_k(s)=\bar{E}^n_k(t) - \bar{E}^n_k(s)$,
\eqref{eq:EincFLLN} implies that
\[
\liminf_{n\to\infty} \bP\left(\Omega_4^n \right)
\ge1-\frac{\eta}{6}.
\]

Next we identify positive constants $\delta$ and $M$, analogous to the
constants $\delta$
and $M$ defined in the proof of Lemma~\ref{prop:NearBoundary}.
Let $\delta= \frac{\varepsilon}{24 \lambda_+}$. Then $\delta<T$.
In order to define $M$, there are two
cases to consider, based on the nature of the abandonment distributions.

\noindent\textit{Case 1}: First suppose that $d_{\max}< \infty$. Let
$w_{\max}(\cdot)$ denote the maximal
workload fluid model solution, i.e., the workload fluid model solution
such that $w_{\max}(0)=d_{\max}$.
By the relative ordering property of workload fluid model solutions, it
follows that for any workload fluid
model solution $w(\cdot)$, $w(t)\le w_{\max}(t)$ for all $t\in\ptime$.
In particular, for any workload fluid
model solution $w(\cdot)$ for all $t\ge\delta$,
%
%
\begin{equation}\label{eq:AfterDelta}
w(t)\le w_{\max}(\delta).
\end{equation}
Set $M_1=(w_{\max}(\delta)+d_{\max})/2$.
By monotonicity properties of workload fluid model
solutions, $w_u < w_{\max}(\delta)<M_1<d_{\max}$.
It follows from \eqref{eq:IC}, (A.1), Theorem~\ref{thrm:JR},
and \eqref{eq:AfterDelta} that
\[
\bP\left( \sup_{t \in[\delta, T]} W^*(t) < M_1 \right) = 1.
\]

\noindent\textit{Case 2}: Next consider the case $d_{\max}=\infty$. By
\eqref{eq:IC} and (A.1),
there exists $M_2 > w_u$ such that
\[
\bP\left( W^*(0) < M_2 \right) \ge1 - \frac{\eta}{6}.
\]
Since $M_2 > w_u$, \eqref{eq:IC}, (A.1), Theorem~\ref{thrm:JR} and
monotonicity properties of workload fluid model solutions imply that
\[
\bP\left( \sup_{t \in[0, T]} W^*(t) < M_2 \right)
=
\bP\left( W^*(0) < M_2 \right)
\ge
1 - \frac{\eta}{6}.
\]
Set
\[
M=
\begin{cases} M_1, &\hbox{if }d_{\max}<\infty, \\ M_2, &\hbox{if
}d_{\max
}=\infty.
\end{cases}
\]

Now we proceed to bound the prelimt processes by $M$ with probability
asymptotically
close to one. For $n\in\N$, let
\[
\Omega_5^n = \left\{ \sup_{t\in[\delta, T]} W^n(t)<M\right\}.
\]
Note that by \eqref{eq:IC}, (A.1), Theorem~\ref{thrm:JR} and that fact
that workload fluid model
solutions are continuous, the convergence in distribution in \eqref
{eq:JR} takes place with
respect to the topology of uniform convergence on compact sets.
Furthermore, the set
$\{ f\in\bD(\ptime, \Rp) : \sup_{t\in[\delta, T]} f(t)<M\}$ is
open with
respect to this topology.
Then by the Portmanteau theorem
\[
\liminf_{n \to\infty}\bP\left( \Omega_5^n\right)\ge1-\frac
{\eta}{6}.
\]
Set
\[
c=\sum_{k=1}^K\rho_k G_k(M)
\qquad\hbox{and}\qquad
\kappa= \min\left( \kappa_0, \kappa_1, \frac{\varepsilon}{24
\lambda
_+}, \frac{\varepsilon c}{72\lambda_+}\right).
\]
Note that $c>0$ since $M<d_{\max}$, and so $\kappa>0$. Finally, for
$n\in\N$, let
\[
\Omega_6^n=\left\{ \sup_{t\in[0, T]} \left| X^n(t)\right| \le
\frac
{\kappa}{2} \right\}.
\]
Then, by \eqref{eq:Xn},
\[
\liminf_{n \to\infty}\bP\left( \Omega_6^n\right)\ge1-\frac
{\eta}{6}.
\]
For $n\in\N$, set
\[
\Omega^n_0 = \Omega^n_1 \cap\Omega^n_2 \cap\Omega^n_3 \cap\Omega
^n_4\cap\Omega_5^n\cap\Omega_6^n.
\]
Then
%
%
\begin{equation} \label{eq:niceOmegas}
\liminf_{n \to\infty}
\bP(\Omega^n_0) \ge1 - \eta.
\end{equation}

Fix $n\in\N$ such that $1/n\le\varepsilon/12$, $1\le k\le K$, $x, y
\in
\Rp$, and $t\in[0, T]$.
The fluid scaled analog of (\ref{eq:jobsInCorner}) is
%
%
\begin{eqnarray} \label{eq:lem:BndReg1}
\bzn_k(t)\left(C_{(x, y)}^{\kappa}\right)
&\le&
\bzn_k(0)\left(C_{(x+t, y+t)}^{\kappa}\right)
+
\frac{1}{n}
\sum_{j=Z^n_k(0)+1}^{A^n_k(t)}
1_{I^\kappa_{x}} \left(w^n_{k, j}(t)\right) \\
&&+ \nonumber
\frac{1}{n}
\sum_{j=Z^n_k(0)+1}^{A^n_k(t)}
1_{I^\kappa_{y}} \left(p^n_{k, j}(t)\right).
\end{eqnarray}
We will show that on $\Omega_0^n$, the right hand side of \eqref
{eq:lem:BndReg1} is
less than or equal to~$\varepsilon$.
Since $\kappa\le\kappa_0$, we have $C_{(x+t, y+t)}^{\kappa}
\subseteq
C_{(x+t, y+t)}^{\kappa_0}$.
Then, on $\Omega^n_1$,
%
%
\begin{equation} \label{eq:lem:BndReg1.1}
\bzn_k(0)\left(C_{(x+t, y+t)}^{\kappa}\right) \le\frac
{\varepsilon}{4}.
\end{equation}
Also, for all $Z_k^n(0)+1\le j\le A_k^n(t)$, the quantity
$p^n_{k, j}(t)$ is either
$a^n_{k, j-Z^n_k(0)}(t)$ of
$v^n_{k, j-Z^n_k(0)}(t)$ (recall \eqref{eq:a} and \eqref{eq:v}). It
follows that, on
$\Omega^n_2 \cap\Omega^n_3$,
%
%
\begin{equation} \label{eq:lem:BndReg1.3}
\frac{1}{n}
\sum_{j=Z^n_k(0)+1}^{A^n_k(t)}
1_{I_y^\kappa}
\left(p^n_{k, j}(t)\right)
\le
\bar{\mathcal{A}}^n_k(t)\left(I_y^\kappa\right)
+
\bar{\mathcal{V}}^n_k(t)\left(I_y^\kappa\right)
\le
\frac{\varepsilon}{2}.
\end{equation}
Combining \eqref{eq:lem:BndReg1}, \eqref{eq:lem:BndReg1.1}, and
\eqref
{eq:lem:BndReg1.3}, we see that on $\Omega_0^n$,
%
%
\begin{equation} \label{eq:BndReg2}
\bzn_k(t)\left(C_{(x, y)}^{\kappa}\right)
\le
\frac{3\varepsilon}{4}
+
\frac{1}{n}\sum_{j=Z^n_k(0)+1}^{A^n_k(t)}
1_{I^\kappa_x} \left(w^n_{k, j}(t)\right).
\end{equation}

Finally, we bound the second term on the right side of \eqref{eq:BndReg2}.
To this end, define the prelimit version $\tau^n(\cdot)$ of $\tau
(\cdot
)$ as follows:
\[
\tau^n(s)=\inf\left\{u \ge0: W^n(u) + u \ge s\right\}, \qquad s
\in
\ptime.
\]
Notice that for all $s\in\ptime$, $W^n(s) + s\ge s$, so that $\tau^n(s)
\in[0, s]$.
During busy periods, $W^n(\cdot) + \iota(\cdot)$ jumps up exactly when
jobs arrive
that contribute to the workload, and remains constant otherwise. During
idle periods,
$W^n(\cdot) + \iota(\cdot)$ increases at rate one. In particular,
$W^n(\cdot) + \iota(\cdot)$
is right continuous and nondecreasing. Then,
%
%
\begin{eqnarray}
W^n(\tau^n(s))+\tau^n(s) &\ge& s, \qquad\hbox{for }s\in\ptime
, \label
{eq:taunprop1}\\
W^n(\tau^n(s)-)+\tau^n(s)&\le&s, \qquad\hbox{for }s\in
(W^n(0), \infty
).\label{eq:taunprop2}
\end{eqnarray}
The function $\tau^n(\cdot)$ will be used to determine which
arrivals may contribute to the sum in \eqref{eq:BndReg2}.
First we show that
%
%
\begin{eqnarray} \label{lem:BndReg2.1}
\frac{1}{n}
\sum_{j=Z^n_k(0)+1}^{A^n_k(t)}
1_{I^\kappa_x} \left(w^n_{k, j}(t)\right)
&\le&
\bar{A}^n_k\left( \tau^n(x+ \kappa+t) \wedge t \right)\\
&&-
\bar{A}^n_k\left( \tau^n( (x-\kappa)^++t) \wedge t \right)+\frac
{1}{n}.\nonumber
\end{eqnarray}
To see this, let $Z^n_k(0)+1\le j\le A^n_k(t)$. Then $t_{k, j}^n\le t$ and
\[
\left( x-\kappa\right)^+ < w_{k, j}^n(t) < x+\kappa
\qquad\Leftrightarrow\qquad
\left( x-\kappa\right)^+ + t < w_{k, j}^n + t_{k, j}^n < x+\kappa+t.
\]
But $w_{k, j}^n =W^n( t_{k, j}^n)$. Then, since $W^n(\cdot)+\iota
(\cdot)$
is nondecreasing,
$t_{k, j}^n\le t$ and $w_{k, j}^n(t)\in I_{x}^{\kappa}$ imply that
\[
\tau^n\left(\left( x-\kappa\right)^++t\right)\wedge t \le
t_{k, j}^n\le
\tau^n(x+\kappa+t)\wedge t.
\]
Therefore, \eqref{lem:BndReg2.1} holds.

Next, we proceed to bound the right hand side of \eqref{lem:BndReg2.1}
on $\Omega_0^n$.
By \eqref{eq:BndReg2}, \eqref{lem:BndReg2.1}, the definition of
$\Omega_4^n$,
and the fact that $n$ is such that $1/n\le\varepsilon/12$,
it suffices to show that on $\Omega_0^n$,
%
%
\begin{eqnarray}\label{eq:TimeChange}
2\lambda_+\left( \tau^n(x+\kappa+t) \wedge t
-
\tau^n( (x-\kappa)^++t) \wedge t \right)\le\frac{\varepsilon}{6}.
\end{eqnarray}
If $\tau^n( (x-\kappa)^++t)\ge t$, then the left side of \eqref
{eq:TimeChange} is zero,
and so \eqref{eq:TimeChange} holds. Henceforth, we assume that $\tau^n(
(x-\kappa)^++t)< t$.
If $ \tau^n(x+\kappa+t) \wedge t\le\delta$, then \eqref
{eq:TimeChange} holds
since $\delta=\varepsilon/24\lambda_+$, which implies that the left
hand side of
\eqref{eq:TimeChange} is no larger than $\varepsilon/12$.
Otherwise, $\delta< \tau^n(x+\kappa+t) \wedge t$
(so that $\tau^n(x+\kappa+t)>0$ and then $x+\kappa+t>W^n(0)$).
First consider the case where $\delta\le\tau^n( (x-\kappa)^++t)$. Then,
by \eqref{eq:taunprop1}, \eqref{eq:taunprop2}, and the nondecreasing
nature of $W^n(\cdot)+\iota(\cdot)$,
\begin{eqnarray*}
2\kappa
&=&x+\kappa+t-(x-\kappa+t)\ge x+\kappa+t-((x-\kappa)^++t)\\
&\ge& W^n(\tau^n(x+\kappa+t)-)+\tau^n(x+\kappa+t)\\
&&-W^n(\tau^n((x-\kappa)^++t))-\tau^n((x-\kappa)^++t)\\
&\ge& W^n(\tau^n(x+\kappa+t)\wedge t-)+\tau^n(x+\kappa+t)\wedge t\\
&&-W^n(\tau^n((x-\kappa)^++t))-\tau^n((x-\kappa)^++t).
\end{eqnarray*}
Using \eqref{eq:WorkRep} and the nondecreasing nature of the idle time
process, we obtain
\begin{eqnarray*}
2\kappa
&\ge& X^n(\tau^n(x+\kappa+t)\wedge t-)-X^n(\tau^n((x-\kappa
)^++t))\\
&& + \sum_{k=1}^K\rho_k\int_{\tau^n((x-\kappa)^++t)}^{\tau
^n(x+\kappa
+t)\wedge t}G_k(W^n(u))\der u.
\end{eqnarray*}
Since $\delta\le\tau^n( (x-\kappa)^++t)$, on $\Omega_5^n\cap
\Omega_6^n$,
\[
2\kappa\ge-\kappa+ c\left(\tau^n(x+\kappa+t)\wedge t-\tau
^n((x-\kappa
)^++t)\right).
\]
By definition of $\kappa$, $3\kappa/c\le\varepsilon/24\lambda_+$,
which implies that the left side of
\eqref{eq:TimeChange} is no larger than $\varepsilon/12$ on $\Omega
_5^n\cap\Omega_6^n$, and
so \eqref{eq:TimeChange} holds.
The last case to consider is the case $\tau^n( (x-\kappa)^++t)<\delta<
\tau^n(x+ \kappa+t) \wedge t$.
Then, by definition of~$\delta$, on $\Omega_4^n$,
%
%
\begin{eqnarray}\label{eq:deltabetween}
&&2\lambda_+\left( \tau^n(x+ \kappa+t) \wedge t
-
\tau^n( (x-\kappa)^++t) \wedge t \right)\\
&&\qquad\le
\frac{\varepsilon}{12}+2\lambda_+\left( \tau^n(x+\kappa+t) \wedge
t-\delta\right).\nonumber
\end{eqnarray}
By \eqref{eq:taunprop1}, \eqref{eq:taunprop2}, and monotonicity of
$W^n(\cdot)+\iota(\cdot)$,
\begin{eqnarray*}
x-\kappa+t &\le& (x-\kappa)^++t\\
&\le& W^n(\tau^n((x-\kappa)^++t))+\tau^n((x-\kappa)^++t)\\
&\le&
W^n(\delta)+\delta\\
&\le&
W^n(\tau^n(x+\kappa+t)\wedge t-)+\tau^n(x+\kappa+t)\wedge t
\le
x+\kappa+t.
\end{eqnarray*}
This together with \eqref{eq:WorkRep} and the nondecreasing nature of
the idle time process
implies that, on $\Omega_5^n\cap\Omega_6^n$,
\begin{eqnarray*}
2\kappa
&\ge& W^n(\tau^n(x+\kappa+t)\wedge t-)+\tau^n(x+\kappa+t)\wedge t-
W^n(\delta)-\delta\\
&\ge& X^n(\tau^n(x+\kappa+t)\wedge t-)-X^n(\delta) \\
&& + \sum_{k=1}^K\rho_k\int_{\delta}^{\tau^n(x+\kappa+t)\wedge t}
G_k(W^n(u))\der u\\
&\ge& -\kappa+ c\left(\tau^n(x+\kappa+t) \wedge t-\delta\right).
\end{eqnarray*}
Using the definition of $\kappa$ and \eqref{eq:deltabetween} implies
\eqref{eq:TimeChange}.
\end{proof}

\subsection{Oscillation Bounds}\label{sec:osc}
This section establishes the second main ingredient for proving
tightness of the fluid scaled residual deadline related processes and state
descriptors, controlled oscillations (see Lemmas~\ref{lem:osc1} and
\ref
{lem:osc}). Both proofs are similar in spirit. The one for the residual
deadline processes is slightly simpler, so it is presented first. For
this, recall the definition of ${\bf d}_K$ given in \eqref{eq:metric}.
Then the modulus of continuity is defined as follows.

\begin{definition} Let $i=1, 2$. For each $\zeta(\cdot)\in\bD
(\ptime
, \M_i^K)$
and each $T>\delta>0$, define the modulus of continuity on $[0, T]$
by
\[
\bw_T(\zeta(\cdot), \delta)=\sup_{t \in[0, T-\delta]}\sup_{h\in
[0, \delta
]} {\bf d}_K[ \zeta(t+h), \zeta(t)].
\]
\end{definition}

\begin{lemma}\label{lem:osc1} Suppose that (AA) holds.
For all $T>0$ and $\varepsilon, \eta\in(0, 1)$, there exists $\delta
\in(0, T)$
such that
\[
\liminf_{n\to\infty} \bP( \bw_T(\CRbn(\cdot), \delta)\le
\varepsilon
)\ge1-\eta.
\]
\end{lemma}

\begin{proof}
Fix $T>0$ and $\varepsilon, \eta\in(0, 1)$.
By \eqref{eq:EFLLN} and Lemma~\ref{lem:BndReg1}, there exists a
$\kappa\in(0, \varepsilon)$ such that for any $\delta\in(0, T)$,
the events
\begin{eqnarray*}
\Omega^n_1
&=&
\left\{ \max_{1\le k \le K}
\sup_{t \in[0, T]} \CRbn_k(t)([0, \kappa]) \le\varepsilon\right\}
, \\
\Omega^n_2
&=&
\left\{\max_{1 \le k \le K}
\sup_{t\in[0, T-\delta]}\sup_{h\in[0, \delta]} \left[\Ebn_k(t+h)
- \Ebn
_k(t)\right] \le2 \lambda_+ \delta\right\}, \\
\Omega^n_0&=&\Omega^n_1 \cap\Omega^n_2,
\end{eqnarray*}
satisfy
%
%
\begin{equation} \label{eq:osc0}
\liminf_{n \to\infty} \bP(\Omega^n_0) \ge1-\eta.
\end{equation}
Fix such a $\kappa$ and set $\delta= \min(\kappa, \frac
{\varepsilon}{2 \lambda_+})$.

Fix $n\in\N$, $1\le k\le K$, and a closed set $B\in\CB_1$.
Let $t\in[0, T-\delta]$, $h\in(0, \delta]$, and $1\le i\le E_k^n(t)$.
If $g_{k, i}^n(t)>h$ or $g_{k, i}^n(t+h)>0$, then
\[
g_{k, i}^n(t)=g_{k, i}^n(t+h)+h.
\]
Then, for $h\in(0, \delta]$, since $h\le\delta\le\kappa
<\varepsilon$,
\begin{eqnarray*}
\CRbn_k(t)(B)&\le& \CRbn_k(t+h)(B^{\varepsilon})+\CRbn
(t)([0, \kappa])\\
\CRbn_k(t+h)(B)&\le& \CRbn_k(t)(B^{\varepsilon})+\Ebn_k(t+h)-\Ebn_k(t).
\end{eqnarray*}
Then, on $\Omega_0^n$, for all $t\in[0, T-\delta]$ and $h\in
[0, \delta]$,
\[
{\bf d}(\CRbn_k(t+h), \CRbn_k(t))<\varepsilon.
\]
This together with the fact that $n\in\N$, $1\le k\le K$, and the
closed set $B\in\CB_2$
were arbitrary and \eqref{eq:osc0} implies the result.
\end{proof}

Next we generalize the preceding argument to prove the following
analogous lemma
for the fluid scaled state descriptors. There are two main
distinctions. One is that
$\kappa<\varepsilon/2$ since in $h>0$ time units a unit atom moves
along a diagonal
path a distance of $\sqrt{2}h$. The other is that \eqref{eq:dynamics2}
has been established.

\begin{lemma}\label{lem:osc}
For all $T>0$ and $\varepsilon, \eta\in(0, 1)$, there exists $\delta
\in(0, T)$
such that
\[
\liminf_{n\to\infty} \bP( \bw_T(\bzn(\cdot), \delta)\le
\varepsilon)\ge
1-\eta.
\]
\end{lemma}

\begin{proof}
Fix $T>0$ and $\varepsilon, \eta\in(0, 1)$.
For each $\kappa>0$, let $\bar{C}^\kappa$ be the closure of
$C^\kappa
$, or
\[
\bar{C}^\kappa= [0, \kappa] \times\Rp\cup\Rp\times[0, \kappa].
\]
By \eqref{eq:EFLLN} and Lemma~\ref{lem:BndReg}, there exists a
$\kappa\in(0, \varepsilon/2)$ such that for any $\delta\in(0, T)$,
the events
\begin{eqnarray*}
\Omega^n_1
&=&
\left\{ \max_{1\le k \le K}
\sup_{t \in[0, T]} \bzn_k(t)(\bar{C}^\kappa) \le\varepsilon
\right
\}
, \\
\Omega^n_2
&=&
\left\{\max_{1 \le k \le K}
\sup_{t\in[0, T-\delta]}\sup_{h\in[0, \delta]} \left[\Ebn_k(t+h)
- \Ebn
_k(t)\right] \le2 \lambda_+ \delta\right\}, \\
\Omega^n_0&=&\Omega^n_1 \cap\Omega^n_2,
\end{eqnarray*}
satisfy
%
%
\begin{equation} \label{eq:lem:osc0}
\liminf_{n \to\infty} \bP(\Omega^n_0) \ge1-\eta.
\end{equation}
Fix such a $\kappa$ and set $\delta= \min(\kappa, \frac
{\varepsilon}{2 \lambda_+})$.

We begin by noting two basic facts that will be used in the proof.
Firstly, for all $B\in\CB_2$ and $h\in[0, \delta]$,
%
%
\begin{equation} \label{eq:lem:osc1.1}
B \subseteq(B^\varepsilon)_h \cup\bar{C}^\kappa.
\end{equation}
To see this, take some $B\in\CB_2$, $h\in[0, \delta]$, and $(w, p)
\in B
\backslash\bar{C}^\kappa$.
By the construction of $\bar{C}^\kappa$, we have $w, p > \kappa\ge
\delta\ge h$.
Because $h \le\delta\le\kappa< \varepsilon/2$, it follows that
$(w-h, p-h) \in B^\varepsilon$
and $(w, p) \in(B^\varepsilon)_h$. In addition, since $\delta
<\varepsilon/2$, we have that
for all $B\in\CB_2$ and $h\in[0, \delta]$,
%
%
\begin{equation} \label{eq:lem:osc1.2}
B_h \subseteq B^\varepsilon.
\end{equation}

Fix $n\in\N$, $1\le k\le K$, and a closed set $B\in\CB_2$. Let
$\check
B=B\setminus C$.
Note that $\check B\in\CB_{2, 0}$ and by \eqref{eq:zk}, $\bzn
_k(t)(B)=\bzn_k(t)(\check B)$
for all $t\in[0, T]$. Then, we can then conclude from \eqref{eq:zk},
\eqref{eq:dynamics2}, and \eqref{eq:lem:osc1.1} that, on $\Omega^n_0$,
for all
$t\in[0, T-\delta]$ and $h\in[0, \delta]$,
%
%
\begin{eqnarray} \nonumber
\bzn_k(t)(B) &=&\bzn_k(t)(\check B)\\
&\le&\nonumber
\bzn_k(t)((\check B^\varepsilon)_h) + \bzn_k(t)( \bar{C}^\kappa) \\
&\le& \nonumber
\bzn_k(t+h)(\check B^\varepsilon) + \varepsilon\\
&\le& \label{eq:lem:osc1}
\bzn_k(t+h)(B^\varepsilon) + \varepsilon.
\end{eqnarray}
Also, by \eqref{eq:zk}, \eqref{eq:dynamics2}, \eqref{eq:lem:osc1.2},
and the fact that $\delta< \frac{\varepsilon}{2\lambda_+}$,
it is true that on $\Omega^n_0$, for all $t\in[0, T-\delta]$ and
$h\in
[0, \delta]$,
%
%
\begin{eqnarray} \nonumber
\bzn_k(t+h) (B) &=& \bzn_k(t+h) (\check B)\\
&\le& \nonumber
\bzn_k(t)(\check B_h) + \bar{E}^n_k(t+h) - \bar{E}^n_k(t) \\
&\le& \nonumber
\bzn_k(t)(B_h) + \varepsilon\\
&\le& \label{eq:lem:osc2}
\bzn_k(t)(B^\varepsilon) + \varepsilon.
\end{eqnarray}
Because $k$ and $B$ were chosen arbitrarily, \eqref{eq:lem:osc1} and
\eqref{eq:lem:osc2}
imply that on $\Omega_0^n$
\[
\bw_T(\bzn(\cdot), \delta)\le\varepsilon.
\]
The result follows from this and \eqref{eq:lem:osc0}.
\end{proof}

\section{Characterization of Limit Points} \label{sec:char}

The main goal of this section is to prove Theorem~\ref{thrm:flt}.
First we prove Lemma~\ref{lem:ResDeadFLLN},
which is then used in the proof of Theorem~\ref{thrm:flt}.

\begin{proof}[Proof of Lemma~\ref{lem:ResDeadFLLN}]
Throughout, we assume that (AA) holds.
Together Lemmas~\ref{lem:cc1} and~\ref{lem:osc1} imply tightness of
$\{\CRbn(\cdot)\}_{n\in\N}$. Let $\bbM\subset\N$ be a strictly increasing
subsequence tending to infinity and $\CRb(\cdot)$ a process such that
as $m\to\infty$,
\[
\CRbm(\cdot)\Rightarrow\CRb(\cdot).
\]
By Lemma~\ref{lem:BndReg1}, $\CRb(t)$ doesn't charge points for all
$t\in\ptime$ almost surely.
By \eqref{eq:EFLLN} and Lemma~\ref{lem:deadFLLN} and the deterministic
nature of those limiting processes,
as $m\to\infty$,
%
%
\begin{equation}\label{eq:JC1}
\left(\CRbm(\cdot), \bar E^m(\cdot), \bar{\mathcal G}^m(\cdot),
\left\langle
\chi
, \bar{\mathcal G}^m(\cdot)\right\rangle \right)
\Rightarrow
\left(\CRb(\cdot), E^*(\cdot), \CD^*(\cdot), \left\langle \chi, \CD
^*(\cdot
)\right\rangle \right).
\end{equation}
Using the Skorohod representation we may assume without loss of
generally that all
random elements are defined on a common probability space $(\Omega
, {\mathcal F}, \bP)$
such that the joint convergence in \eqref{eq:JC1} is almost sure.
Fix $\omega\in\Omega$ such that $\CRb(t)(\omega)$ doesn't charge points
for all $t\in\ptime$
and as $m\to\infty$,
%
%
\begin{eqnarray}\label{eq:as1}
&&
\left(\CRbm(\cdot)(\omega), \bar E^m(\cdot)(\omega), \bar
{\mathcal
G}^m(\cdot)(\omega), \left\langle \chi, \bar{\mathcal G}^m(\cdot)\right
>(\omega
)\right
)\\
&&\qquad\rightarrow
\left(\CRb(\cdot)(\omega), E^*(\cdot), \CD^*(\cdot), \left\langle \chi
, \CD
^*(\cdot
)\right\rangle \right).\nonumber
\end{eqnarray}
Henceforth, all random variables are evaluated at this $\omega$. It
suffices to show that
$\CRb(\cdot)=\CR^*(\cdot)$.
For this, it suffices to show that for all $1\le k\le K$, $t\in\ptime$
and $x\in\Rp$,
%
%
\begin{equation}\label{eq:R*}
\CRb_k(t)(x, \infty)=\lambda_k\int_0^t G_k(x+t-s)\der s.
\end{equation}
To see this, note that for all $1\le k\le K$, $t\in\ptime$, and $x\in
\Rp$
\begin{eqnarray*}
\CR_k^*(t)(x, \infty)
&=&
\lambda_k\int_x^{\infty} \left[G_k(y)-G_k(y+t)\right]\der y\\
&=&
\lambda_k\int_x^{x+t} G_k(y)\der y
=
\lambda_k\int_0^t G_k(x+t-s)\der s.
\end{eqnarray*}

Next we verify \eqref{eq:R*}. For this, fixed $1\le k\le K$, $t\in
\ptime
$, and $x\in\Rp$.
Given $L\in\N$, let $\kappa=t/L$ and set
\[
t_{\ell} = \ell\kappa\qquad\hbox{and}\qquad x_{\ell}=x+t-t_{\ell
}, \qquad\hbox{for }\ell=0, \dots, L.
\]
Then, given $L\in\N$, $1\le i\le E_k^m(t)$ if and only if there exists
$\ell\in\{0, \dots, L-1\}$ such that
$t_{k, Z_k^m(0)+i}^m\in(t_{\ell}, t_{\ell+1}]$. Further, given $L\in
\N$,
$1\le i\le E_k^m(t)$ and $\ell\in\{0, \dots, L-1\}$
such that $t_{k, Z_k^m(0)+i}^m\in(t_{\ell}, t_{\ell+1}]$, $x_{\ell
}<g_{k, i}^m$ implies that $x<g_{k, i}^m(t)$. Similarly,
given $L\in\N$, $1\le i\le E_k^m(t)$ and $\ell\in\{0, \dots, L-1\}
$ such
that $t_{k, Z_k^m(0)+i}^m\in(t_{\ell}, t_{\ell+1}]$,
$x<g_{k, i}^m(t)$ implies that $x_{\ell+1}<g_{k, i}^m$. Hence, given
$L\in
\N$,
\[
\sum_{\ell=0}^{L-1} \CGbm_k\left(t_{\ell}, t_{\ell+1}\right
)\left
(x_{\ell
}, \infty\right)
\le
\CRbm_k(t)(x, \infty)
\le
\sum_{\ell=0}^{L-1} \CGbm_k\left(t_{\ell}, t_{\ell+1}\right
)\left
(x_{\ell
+1}, \infty\right).
\]
By \eqref{eq:as1}, given $L\in\N$ and $\varepsilon>0$ there exists
$M\in
\N$
such that for all $m\ge M$
\begin{eqnarray*}
\max_{0\le\ell\le L-1} \left|
\CGbm_k\left(t_{\ell}, t_{\ell+1}\right)\left(x_{\ell}, \infty
\right)
-
\CD_k^*\left(t_{\ell}, t_{\ell+1}\right)\left(x_{\ell}, \infty
\right)
\right|
&<&\frac{\varepsilon}{2L}, \\
\max_{0\le\ell\le L-1} \left|
\CGbm_k\left(t_{\ell}, t_{\ell+1}\right)\left(x_{\ell+1}, \infty
\right)
-
\CD^*_k\left(t_{\ell}, t_{\ell+1}\right)\left(x_{\ell+1}, \infty
\right)
\right|
&<&\frac{\varepsilon}{2L}.
\end{eqnarray*}
Hence, given $L\in\N$ and $\varepsilon>0$ there exists $M\in\N$
such that for all $m\ge M$,
\[
\sum_{\ell=0}^{L-1} \CD_k^*\left(t_{\ell}, t_{\ell+1}\right
)\left
(x_{\ell
}, \infty\right)-\frac{\varepsilon}{2}
\le
\CRbm_k(t)(x, \infty)
\le
\sum_{\ell=0}^{L-1} \CD_k^*\left(t_{\ell}, t_{\ell+1}\right
)\left
(x_{\ell
+1}, \infty\right)+\frac{\varepsilon}{2}.
\]
Given $L\in\N$, we have that
\begin{eqnarray*}
\sum_{\ell=0}^{L-1} \CD_k^*\left(t_{\ell}, t_{\ell+1}\right
)\left
(x_{\ell
}, \infty\right)
&=&
\sum_{\ell=0}^{L-1}\lambda_k\kappa G_k\left(x_{\ell}\right)\\
\sum_{\ell=0}^{L-1} \CD_k^*\left(t_{\ell}, t_{\ell+1}\right
)\left
(x_{\ell
+1}, \infty\right)
&=&
\sum_{\ell=0}^{L-1}\lambda_k\kappa G_k\left(x_{\ell+1}\right).
\end{eqnarray*}
Respectively these are upper and lower Riemann sums, and since
$G_k(\cdot)$ is continuous, they both converge to
$\lambda_k\int_0^t G_k(x+t-s)\der s$ as $L\to\infty$. Given
$\varepsilon
>0$, let $\hat L\in\N$ be such that
\[
\left| \sum_{\ell=0}^{\hat L-1}\lambda_k\kappa G_k\left(x_{\ell
+1}\right
)- \sum_{\ell=0}^{\hat L-1}\lambda_k\kappa G_k\left(x_{\ell}\right
)\right|
<\frac{\varepsilon}{2}.
\]
Hence, given $\varepsilon>0$, there exists $\hat M\in\N$ such that for
all $m\ge\hat M$,
\[
\left|\CRbm_k(t)(x, \infty)-\lambda_k\int_0^t G_k(x+t-s)\der
s\right
|<\varepsilon.
\]
Thus, \eqref{eq:R*} holds, as desired.
\end{proof}

Having proved Lemma~\ref{lem:ResDeadFLLN}, we are now ready to prove
Theorem~\ref{thrm:flt}.

\begin{proof}[Proof of Theorem~\ref{thrm:flt}]
Together Lemmas~\ref{lem:compact containment} and~\ref{lem:osc} imply
tightness of
$\{\bzn(\cdot)\}_{n\in\N}$. Let $\bbM\subset\N$ be a strictly
increasing subsequence
tending to infinity and $\z^*(\cdot)$ a process such that as $m\to
\infty$,
%
%
\begin{equation}\label{eq:SubSeq}
\bzm(\cdot)\Rightarrow\z^*(\cdot).
\end{equation}
Note that by (A.1), $\z^*(\cdot)\in\I$ almost surely. Hence, in order
to prove Theorem~\ref{thrm:flt},
it suffices to show that $\z^*(\cdot)$ satisfies \eqref{fdevoeq2}
almost surely. Indeed, once this
is verified, it follows by the uniqueness asserted in Theorem \ref
{thrm:eu} that the law of the limit
point $\z^*(\cdot)$ is unique, and so $\bzn(\cdot)\Rightarrow\z
^*(\cdot
)$ as $n\to\infty$.

By \eqref{eq:IC}, (A.1), and Theorem~\ref{thrm:JR}, as $m\to\infty$,
%
%
\begin{equation}\label{eq:JR2}
W^m(\cdot)\Rightarrow W^*(\cdot),
\end{equation}
where $W^*(\cdot)$ is almost surely a workload fluid model solution
such that $W^*(0)$
is equal in distribution to $W_0^*$.
We would like to argue that this convergence is joint with \eqref{eq:SubSeq}.
By \eqref{eq:IC}, \eqref{eq:Xn}, \eqref{eq:A}, \eqref{eq:V},
and the fact that the limit in \eqref{eq:Xn} and $\CR^*(\cdot)$ are
deterministic,
as $m\to\infty$,
\[
\left(\bzm(\cdot), W^m(0), X^m(\cdot), \bar{\mathcal A}^m(\cdot
), \bar
{\mathcal V}^m(\cdot)\right)
\Rightarrow
\left(\z^*(\cdot), W_0, 0, \CR^*(\cdot), \CR^*(\cdot)\right).
\]
This together with \eqref{eq:CMT}, \eqref{eq:CMT2}, and \eqref{eq:JR2}
implies that
as $m\to\infty$,
%
%
\begin{eqnarray}\label{eq:JointConvergence}
&&\left(\bzm(\cdot), W^m(\cdot), X^m(\cdot), \bar{\mathcal
A}^m(\cdot), \bar
{\mathcal V}^m(\cdot)\right)\\
&&\qquad\Rightarrow
\left(\z^*(\cdot), W^*(\cdot), 0, \CR^*(\cdot), \CR^*(\cdot
)\right
).\nonumber
\end{eqnarray}

Using the Skorohod representation we may assume without loss of
generally that all
random elements are defined on a common probability space $(\Omega
, {\mathcal F}, \bP)$
such that the joint convergence in \eqref{eq:JointConvergence} is
almost sure.
By \eqref{eq:IC}, $\z^*(0)$ satisfies $w_{\z^*(0)}=W^*(0)$, (A.1),
(A.2), and (A.3) almost surely.
Furthermore, by Lemma~\ref{lem:osc}, $\z^*(\cdot)$ is continuous
almost surely.
In addition, by Lemma~\ref{lem:BndReg},
%
%
\begin{equation}\label{eq:NoCornerMass}
\bP\left(\z_+^*(t)(C_x)=0\hbox{ for all }t\in\ptime\hbox{ and
}x\in\Rp
^2\right)=1.
\end{equation}
(cf.\ \cite[Lemma 6.2] {ref:GRZ}).
Fix $\omega\in\Omega$ such that $W^*(\cdot)(\omega)$ is a workload
fluid model solution
and $\z^*(\cdot)(\omega)$ is continuous and
satisfies \eqref{eq:NoCornerMass}, $w_{\z^*(0)(\omega
)}=W^*(0)(\omega
)$, (A.1), (A.2), (A.3)
and, as $m\to\infty$,
\begin{eqnarray*}
&&\left(\bzm(\cdot)(\omega), W^m(\cdot)(\omega), \bar{\mathcal
A}^m(\cdot
)(\omega), \bar{\mathcal V}^m(\cdot)(\omega)\right)\\
&&\qquad\rightarrow\left(\z^*(\cdot)(\omega), W^*(\cdot)(\omega
), \CR
^*(\cdot), \CR^*(\cdot)\right).
\end{eqnarray*}
For all $t\in\ptime$ and $m\in\bbM$, set
\begin{eqnarray*}
\zeta^m(t)=\bzm(t)(\omega)&\qquad\hbox{and}\qquad&
w^m(t)=W^m(t)(\omega
), \\
\zeta(t)=\z^*(t)(\omega)&\qquad\hbox{and}\qquad&
w(t)=W^*(t)(\omega).
\end{eqnarray*}
By \eqref{eq:IC}, $w(0)=w_{\vartheta}$ where $\vartheta=\zeta(0)$.
In order to prove Theorem~\ref{thrm:flt}, it suffices to show that
$\zeta(\cdot)$ is a fluid model solution for the supercritical data
$(\lambda, \mu, \Gamma)$ and initial measure $\vartheta$. In particular,
we must show that $\zeta(\cdot)$ satisfies \eqref{fdevoeq2}. In this
regard, recall that ${\mathcal P}$ is a $\pi$-system (see \eqref
{def:Pi}). As in the proof of Theorem~\ref{thrm:eu} given in Section
\ref{sec:prfthrmeu}, it is enough to show that $\zeta(\cdot)$ satisfies
\eqref{fdevoeq2} for all $B\in{\mathcal P}$.

Fix $B\in{\mathcal P}$. Then $B=[a, \infty)\times[c, \infty)$ for some
$0\le a, c<\infty$.
Fix $t\in\ptime$. In what follows, all random elements are evaluated at
the specific $\omega$
fixed in the preceding paragraph.
Since $\zeta(\cdot)$ is continuous, $\zeta^n(s)\wk\zeta(s)$ for all
$s\in[0, t]$.
By \eqref{eq:NoCornerMass}, $\zeta_+(0)(C_{(a+t, c+t)})=0$ and $\zeta
_+(t)(C_{(a, c)})=0$. Then,
for each $1\le k\le K$, we have
\[
\lim_{m\to\infty}\zeta_k^m(0)(B_t)=\zeta_k(0)(B_t)
\qquad\hbox{and}\qquad
\lim_{m\to\infty}\zeta_k^m(t)(B)=\zeta_k(t)(B).
\]
However, by the fluid scaled version of \eqref{eq:dynamics2} with $h=t$
and $t=0$,
for each $m\in\bbM$ and $1\le k\le K$,
\begin{eqnarray*}
\zeta_k^m(t)(B)-\zeta_k^m(0)(B_t)
&=&
\frac{1}{m}\sum_{j=Z_k^m(0)+1}^{A_k^m(t)} 1_{B_{t-t_{k, j}^m}}\left
(w_{k, j}^m, p_{k, j}^m\right)\\
&=&
\frac{1}{m}\int_0^{t} 1_{B_{t-s}}\left(w^m(s), p_{k,
A_k^m(s)}^m\right)
\der E_k^m(s).
\end{eqnarray*}
Hence, in order to verify that \eqref{fdevoeq2} holds, it suffices to
show that for each $1\le k\le K$,
\[
\lim_{m\to\infty}\frac{1}{m}\int_0^{t} 1_{B_{t-s}}\left
(w^m(s), p_{k, A_k^m(s)}^m\right)
\der E_k^m(s)
= \lambda_k\int_0^t\left( \delta_{w(s)}^+\times\Gamma_k\right
)(B_{t-s})\der s.
\]
Note that for $1\le k\le K$,
\begin{eqnarray*}
\lambda_k\int_0^t\left( \delta_{w(s)}^+\times\Gamma_k\right
)(B_{t-s})\der s
&=& \lambda_k
\int_0^t
1_{\{w(s) \ge a + t-s\}} G_k(c+t-s)\der s \\
&=& \lambda_k
\int_{\tau(a+t)\wedge t}^t
G_k(c+t-s)\der s \\
&=& \lambda_k
\int_c^{c+t-\tau(a+t)\wedge t}
G_k(u)\der u.
\end{eqnarray*}
Hence, in order to verify that \eqref{fdevoeq2} holds, it suffices to
show that for each $1\le k\le K$,
%
%
\begin{eqnarray}\label{eq:MainPart}
&& \lim_{m\to\infty}\frac{1}{m}\int_0^{t} 1_{B_{t-s}}\left
(w^m(s), p_{k, A_k^m(s)}^m\right) \der E_k^m(s)\\
&&\qquad=\lambda_k
\int_c^{c+t-\tau(a+t)\wedge t}
G_k(u)\der u.\nonumber
\end{eqnarray}

In order to verify \eqref{eq:MainPart}, fix $1\le k\le K$.
Note that for all $m\in\bbM$ and $1\le i \le E_k^m(t)$,
\[
d_{k, i}
\le
p_{k, Z_k^m(0)+i}^m
\le
d_{k, i}+v_{k, i}^m.
\]
Then
\begin{eqnarray*}
&&\int_0^{t} 1_{B_{t-s}}\left(w^m(s), d_{k, E_k^m(s)}\right) \der
E_k^m(s)\\
&&\qquad\le
\int_0^{t} 1_{B_{t-s}}\left(w^m(s), p_{k, A_k^m(s)}^m\right) \der
E_k^m(s)\\
&&\qquad\le
\int_0^{t} 1_{B_{t-s}}\left
(w^m(s), d_{k, E_k^m(s)}+v_{k, E_k^m(s)}^m\right
) \der E_k^m(s).
\end{eqnarray*}
Let $0<\delta<t$. Take $M$ such that for all $m\ge M$, $\sup_{0\le
s\le t}|w^m(s)-w(s)|<\delta$.
Then, for all $m\ge M$,
%
%
\begin{eqnarray}
&&\int_0^{t} 1_{B_{t-s}}\left(w(s)-\delta, d_{k, E_k^m(s)}\right)
\der
E_k^m(s)\nonumber\\
&&\qquad\le
\int_0^{t} 1_{B_{t-s}}\left(w^m(s), p_{k, A_k^m(s)}^m\right) \der
E_k^m(s)\label{eq:center}\\
&&\qquad\le
\int_0^{t} 1_{B_{t-s}}\left(w(s)+\delta
, d_{k, E_k^m(s)}+v_{k, E_k^m(s)}^m\right) \der E_k^m(s).\nonumber
\end{eqnarray}
For $s\in\ptime$, $w(s)-\delta\ge a+t-s$ if and only if $s\ge\tau
(a+t+\delta)$. Then
\begin{eqnarray*}
&&\int_0^{t} 1_{B_{t-s}}\left(w(s)-\delta, d_{k, E_k^m(s)}\right)
\der
E_k^m(s)\\
&&\quad=
\int_{\tau(a+t+\delta)\wedge t} ^{t} 1_{[c+t-s, \infty)}\left
(d_{k, E_k^m(s)}\right) \der E_k^m(s)\\
&&\quad=
\int_{\tau(a+t+\delta)\wedge t} ^{t} 1_{[c, \infty)}\left(d_{k, E_k^m(s)}
- (t-s)\right) \der E_k^m(s)\\
&&\quad=
\int_{0} ^{t} 1_{[c, \infty)}\left(a_{k, E_k^m(s)}^m(t) \right)
\der
E_k^m(s)\\
&&\qquad-
\int_0^{\tau(a+t+\delta)\wedge t} 1_{[c, \infty)}\left
(a_{k, E_k^m(s)}^m(t)\right) \der E_k^m(s)\\
&&\quad=
{\mathcal A}_k^m(t)([c, \infty))
-{\mathcal A}_k^m(\tau(a+t+\delta)\wedge t)([c+t-\tau(a+t+\delta
)\wedge
t, \infty)).
\end{eqnarray*}
Similarly,
\begin{eqnarray*}
&&\int_0^{t} 1_{B_{t-s}}\left(w(s)+\delta
, d_{k, E_k^m(s)}+v_{k, E_k^m(s)}^m\right) \der E_k^m(s)\\
&&\quad=
{\mathcal V}_k^m(t)([c, \infty))
-{\mathcal V}_k^m(\tau(a+t-\delta)\wedge t)([c+t-\tau(a+t-\delta
)\wedge
t, \infty)).
\end{eqnarray*}
Then, by \eqref{eq:A} and \eqref{eq:center},
\begin{eqnarray*}
&&\liminf_{m\to\infty}\frac{1}{m}\int_0^{t} 1_{B_{t-s}}\left
(w^m(s), p_{k, A_k^m(s)}^m\right) \der E_k^m(s)\\
&&\quad\ge
\CR_k^*(t)([c, \infty))-\CR_k^*(\tau(a+t+\delta)\wedge
t)([c+t-\tau
(a+t+\delta)\wedge t, \infty)).
\end{eqnarray*}
But,
\begin{eqnarray*}
&&\CR_k^*(t)([c, \infty))
-
\CR_k^*(\tau(a+t+\delta)\wedge t )([c+t-\tau(a+t+\delta)\wedge
t, \infty
))\\
&&\quad= \lambda_k \int_c^{\infty}\left( G_k(u)-G_k(u+t)\right
)\der
u\\
&&\qquad-\lambda_k \int_{ c+t-\tau(a+t+\delta)\wedge
t}^{\infty
}\left( G_k(u)-G_k(u+\tau(a+t)\wedge t )\right)\der u\\
&& \quad= \lambda_k
\int_c^{c+t-\tau(a+t+\delta)\wedge t}
G_k(u)\der u.
\end{eqnarray*}
So then, since $\delta\in(0, t)$ is arbitrary, letting $\delta
\searrow
0$ and using continuity of $\tau(\cdot)$
yields that
\[
\liminf_{m\to\infty}\frac{1}{m}\int_0^{t} 1_{B_{t-s}}\left
(w^m(s), p_{k, A_k^m(s)}^m\right) \der E_k^m(s)
\ge
\lambda_k
\int_c^{c+t-\tau(a+t)\wedge t}
G_k(u)\der u.
\]
Similarly, using \eqref{eq:V} in place of \eqref{eq:A},
\[
\limsup_{m\to\infty} \frac{1}{m}\int_0^{t} 1_{B_{t-s}}\left
(w^m(s), p_{k, A_k^m(s)}^m\right) \der E_k^m(s)
\le
\lambda_k \int_c^{c+t-\tau(a+t)\wedge t} G_k(u)\der u.
\]
Hence \eqref{eq:MainPart} holds, as desired.
\end{proof}

\bibliographystyle{acmtrans-ims}

%
\end{document}